\documentclass{amsart}

\usepackage{amssymb}
\usepackage{marvosym}
\usepackage{epsfig}
\usepackage{fullpage}

\usepackage{amscd} 
\usepackage{epstopdf}
\usepackage{graphicx}
\usepackage{color}
\usepackage{mathptmx}   
\DeclareMathAlphabet{\mathcal}{OMS}{cmsy}{m}{n} 
\usepackage{hyperref}

\theoremstyle{plain}
\newtheorem{theorem}{Theorem}[section]    
\newtheorem{conjecture}[theorem]{Conjecture}
\newtheorem{proposition}[theorem]{Proposition}
\newtheorem{corollary}[theorem]{Corollary}

\newtheorem{lemma}[theorem]{Lemma}

\theoremstyle{definition}
\newtheorem{remark}[theorem]{Remark}

\newtheorem{example}[theorem]{Example}
\newtheorem{question}[theorem]{Question}
    
\newcommand{\Z}{\mathbb Z}

\newcommand{\Q}{\mathbb Q}

\newcommand{\bdry}{\partial}

\newcommand{\st}[1]{\operatorname{st}}

\newcommand{\nbhd}{\mathcal{N}}

\renewcommand{\hat}[1]{\widehat{{#1}}}

\newcommand{\wind}{{\mathrm{wind}}}
\newcommand{\wrap}{{\mathrm{wrap}}}

\renewcommand{\(}{\textup{(}}
\renewcommand{\)}{\textup{)}}

\theoremstyle{plain}
\newtheorem*{claim*}{Claim}
\newtheorem*{thm:slicegenera_seifertgenera}{Theorem~\ref{thm:slicegenera_seifertgenera}}
\newtheorem*{thm:coherent}{Theorem~\ref{thm:coherent}}
\newtheorem*{thm:converse}{Theorem~\ref{converse}}
\newtheorem*{thm:Murasugi-sum}{Theorem~\ref{thm:Murasugi-sum}}
\newtheorem*{thm:positive_L-space_knot}{Theorem~\ref{thm:postwistposLspace}}
\newtheorem*{cor:knots=>surgeries}{Corollary~\ref{knots=>surgeries}}
\newtheorem*{thm:limit}{Theorem~\ref{thm:limit}}
\newtheorem*{thm:postwistquasipositive}{Theorem~\ref{thm:postwistquasipositive}}  
\newtheorem*{theorem*}{Theorem}

\newtheorem*{thm:positivetwistLspaceknots}{Theorem~\ref{thm:positivetwistLspaceknots}}
\newtheorem*{thm:slice_genus_bound}{Theorem~\ref{thm:slice_genus_bound}}
\newtheorem*{thm:values}{Theorem~\ref{thm:values}}
\newtheorem*{thm:braidaxis}{Theorem~\ref{thm:braidaxis}}

\begin{document}
\baselineskip 14pt

\title{Seifert vs slice genera of knots in twist families and a characterization of braid axes}

\author{Kenneth L. Baker and Kimihiko Motegi}

\address{Department of Mathematics, University of Miami, 
Coral Gables, FL 33146, USA}
\email{k.baker@math.miami.edu}
\address{Department of Mathematics, Nihon University, 
3-25-40 Sakurajosui, Setagaya-ku, 
Tokyo 156--8550, Japan}
\email{motegi@math.chs.nihon-u.ac.jp}

\thanks{The first named author was partially supported by a grant from the Simons Foundation (\#209184 to Kenneth L.\ Baker).  
The second named author was partially supported by JSPS KAKENHI Grant Number JP26400099 and Joint Research Grant of Institute of Natural Sciences at Nihon University for 2017.  
}

\dedicatory{}

\begin{abstract}
Twisting a knot $K$ in $S^3$  along a disjoint unknot $c$ produces a twist family of knots $\{K_n\}$ indexed by the integers.
Comparing the behaviors of the Seifert genus $g(K_n)$ and the slice genus $g_4(K_n)$ under twistings, 
we prove that if $g(K_n) - g_4(K_n) < C$ for some constant $C$ for infinitely many integers $n > 0$
or $g(K_n) / g_4(K_n) \to 1$ as $n \to \infty$, 
then either the winding number of $K$ about $c$ is zero or the winding number equals the wrapping number. 
As a key application, 
if $\{K_n\}$ or the mirror twist family $\{\overline{K_n}\}$ 
contains infinitely many tight fibered knots, then the latter must occur.
We further develop this to show that $c$ is a braid axis of $K$ if and only if both $\{K_n\}$ and $\{\overline{K_n}\}$ each contain infinitely many tight fibered knots.
We also give a necessary and sufficient condition for $\{ K_n \}$ to contain infinitely many L-space knots, 
and show (modulo a conjecture) that satellite L-space knots are braided satellites. 
\end{abstract}

\maketitle

\date{\today}

{
\renewcommand{\thefootnote}{}
}

\tableofcontents

\section{Introduction}
\label{section:Introduction}

Let $K$ be a knot in the $3$--sphere $S^3$.  
The {\em (Seifert) genus} of $K$, $g(K)$, is the minimal genus of a connected, orientable surface that is embedded in $S^3$ and bounded by $K$.  
Regarding $S^3$ as the boundary of the $4$--ball $B^4$, 
the {\em (smooth) slice genus} of $K$, $g_4(K)$ is the minimal genus of a connected, orientable surface that is smoothly, 
properly embedded in $B^4$ and bounded by $K$.  
Observe that $g_4(K) \leq g(K)$.

Now take an unknot $c$ in $S^3$ disjoint from $K$.  
The {\em winding number} of $K$ about $c$, $\omega= \wind_c(K)=|lk(K,c)|$, 
is the absolute value of the algebraic intersection number of $K$ and a disk bounded by $c$, (i.e.\ the absolute value of the linking number).  
The {\em wrapping number} of $K$ about $c$, $\eta= \wrap_c(K)$, 
is the minimal geometric intersection number of $K$ and a disk bounded by $c$.  
Observe that $\wind_c(K) \leq \wrap_c(K)$.  

Performing $(-1/n)$--surgery on $c$ causes $K$ to twist $n$ times along a disk bounded by $c$, 
producing a {\em twist family} of knots $\{K_n\}_{n \in \Z}$.
When $\wind_c(K)=\wrap_c(K)$ there is a disk bounded by $c$ that $K$ intersects in a single direction.  
In this case we say that $c$ links $K$ {\em coherently} and that $\{K_n\}$ is a {\em coherent twist family}.

 When $\wrap_c(K) \leq 1$, then $c$ is either a meridian of $K$ or split from $K$ so that $K=K_n$ for all $n$.   
However, when $\wrap_c(K) >1$, then $\{ K_n \}$ contains infinitely many distinct knots, 
and $g(K_n)=g_4(K_n)=0$ occurs for at most two integers $n$; 
see \cite{KMS, Mathieu}.

\subsection{Comparison of Seifert genera and slice genera of knots under twisting} 

Since $g(K) \geq g_4(K) \geq 0$ for any knot $K$,  
the discrepancy between the Seifert and slice genera of knots in twist families may be measured in terms of the behavior of their ratio $g(K_n)/g_4(K_n)$. 
(Regard $N/0$ as $\infty$ for numbers $N>0$.)

\begin{theorem}\
\label{thm:values}
\begin{enumerate}
\item
For any twist family $\{K_n\}$ with $\omega > 0$, 
$\displaystyle \lim_{n \to \infty} \frac{g( K_n )}{g_4(K_n )} = r$ where $1 \le r \in \mathbb{Q} \cup \{\infty\}$.
\item
For any $1 \le r \in \mathbb{Q} \cup \{\infty\}$, 
there exists a twist family $\{ K_n \}$ with $\omega>0$ which satisfies 
$\displaystyle \lim_{n \to \infty} \frac{g( K_n )}{g_4(K_n )} = r$. 
\end{enumerate}
\end{theorem}

We will establish this theorem in Section~\ref{examples}. 
The proof of the first assertion requires an understanding of behaviors of both Seifert genera and slice genera under twisting. 
For the second assertion we will give explicit examples with the desired property.

In the case that $\omega=0$,  
Proposition~\ref{prop:w=0} demonstrates that Theorem~\ref{thm:values}(2) still holds. 
Considering Theorem~\ref{thm:values}(1) when $\omega=0$, 
Question~\ref{question_w=0} asks if the limit $\displaystyle \lim_{n \to \infty} \frac{g( K_n )}{g_4(K_n )}$ is always defined.

\medskip
When the Seifert genus and the slice genus are asymptotically same 
(or alternatively when their difference is bounded infinitely often) we find a strong restriction on the linking between $c$ and $K$. 

\begin{theorem}
\label{thm:slicegenera_seifertgenera}
Let $\{ K_n \}$ be a twist family of knots obtained by twisting $K$ along $c$. 

If either 
\begin{enumerate}
\item there exists a constant $C \geq 0$ such that $g(K_n)-g_4(K_n) \leq C$ for infinitely many integers $n$, or
\item $g(K_n) / g_4(K_n) \to 1$ as $n \to \infty$ or $n \to -\infty$,
\end{enumerate}
then either $\wind_c(K)= 0$ or $c$ links $K$ coherently. 
\end{theorem}

For winding number $0$ twist families, since $g(K_n)$ and hence also $g_4(K_n)$ are bounded above, 
the converse to Theorem~\ref{thm:slicegenera_seifertgenera}(1) holds.  
However it is simple to create examples of winding number $0$ twist families for which the converse of Theorem~\ref{thm:slicegenera_seifertgenera}(2) fails, 
see Proposition~\ref{prop:w=0}.  
On the other hand, for coherent twist families, we have a full converse when we discard trivialities.   
If $c$ links $K$ coherently at most once (so that $c$ is either a meridian or split from $K$), then $K_n=K$ for all $n$ and the converse to Theorem~\ref{thm:slicegenera_seifertgenera}(1) always holds with $C=g(K)-g_4(K)$ 
while the converse to (2) holds only when $g(K)=g_4(K)$.

\begin{theorem}
\label{thm:coherent}
Let $\{ K_n \}$ be a twist family of knots obtained by twisting $K$ along $c$. 

If $c$ links $K$ coherently at least twice, then both
\begin{enumerate}
\item there exists constants $C_+,C_- \geq 0$ such that $g(K_n)-g_4(K_n) = C_\pm$ for sufficiently large integers $\pm n$, and hence
\item $g(K_n) / g_4(K_n) \to 1$ as $|n| \to \infty$.
\end{enumerate}
\end{theorem}

Furthermore, when $c$ links $K$ coherently,
the knots $K_n$ with $n$ sufficiently large have minimal genus Seifert surfaces that are the Murasugi sum of a fixed surface with fibers of torus links.

\begin{theorem}
\label{thm:Murasugi-sum}
Let $\{K_n\}$ be a twist family in which $c$ links $K$ coherently $\omega>0$ times. 
Then there is an integer $n_0$ such that for all $n > n_0$, 
a minimal genus Seifert surface of $K_n$ may be obtained as a Murasugi sum of a minimal genus Seifert surface for $K_{n_0}$ and the fiber surface for the $(\omega, (n-n_0) \omega + 1)$--torus knot. 
\end{theorem}

\subsection{Strongly quasipositive knots and tight fibered knots}

Briefly, a knot in $S^3$ is {\em strongly quasipositive} if it is the boundary of a {\em quasipositive} Seifert surface, 
a special kind of Seifert surface obtained from parallel disks by attaching positive bands in a particular kind of braided manner; 
see the leftmost picture in Figure~\ref{quasipositive} for a quick reminder or \cite{Rudolph_HandBook} for  more details and  context.   
We say a fibered knot in $S^3$ is {\em tight} if, 
as an open book for $S^3$, it supports the positive tight contact structure on $S^3$.   
Hedden proved that tight fibered knots are precisely the fibered, strongly quasipositive knots \cite[Proposition~2.1]{Hed_positive} (see also \cite[Theorem~3.1]{BI}).

The above theorems allow us to determine structure in twist families for which the Seifert genus equals the slice genus for infinitely many of the knots. For instance, 
Theorem~\ref{thm:slicegenera_seifertgenera} almost immediately yields the following corollary about strongly quasipositive knots and tight fibered knots 
(i.e.\ fibered strongly quasipositive knots).

\begin{corollary}
\label{cor:sqptwist}
Let $\{ K_n \}$ be a twist family of knots obtained by twisting $K$ along $c$ 
such that $c$ is neither split from $K$ nor a meridian of $K$.
\begin{enumerate}
\item If  $\{ K_n \}$ or its mirror $\{ \overline{K}_n \}$ contains infinitely many strongly quasipositive knots, 
then either $\wind_c(K)=0$ or $c$ links $K$ coherently.
\item If  $\{ K_n \}$ or its mirror $\{ \overline{K}_n \}$ contains infinitely many tight fibered knots, 
then $c$ links $K$ coherently.
\end{enumerate}
\end{corollary}

\begin{proof}
If $K$ is a strongly quasipositive knot $K$, 
then $g(K) = g_4(K)$ \cite[Theorem~4]{Liv}, 
and (1) immediately follows from Theorem~\ref{thm:slicegenera_seifertgenera}. 
For the second assertion, 
in addition to Theorem~\ref{thm:slicegenera_seifertgenera}, 
\cite[Theorem~3.1]{BM} is also needed to show $\wind_c(K) \neq 0$ 
when the twist family contains infinitely many tight fibered knots.
\end{proof}

\begin{example}
\label{w=0_twist_quasipositive}
Let $K$ be a strongly quasipositive knot which bounds quasipositive Seifert surface $F$.  
Take an unknot $c$ as in Figure~\ref{quasipositive} which bounds a disk intersecting one of the positive bands of $F$ in a spanning arc. 
Then $\wind_c(K)= 0$, 
but $n$--twisting along $c$ makes $F$ into a quasipositive Seifert surface of $K_n$ for any non-positive integer $n$; 
see Figure~\ref{quasipositive}. 
Thus $K_n$ is a strongly quasipositive knot if $n \leq 0$. 
One may construct further examples using Rudolph's  ``braidzel'' surfaces \cite{Rudolph_braidzel}. 
\end{example}

\begin{figure}[h]
\begin{center}
\includegraphics[width=0.8\linewidth]{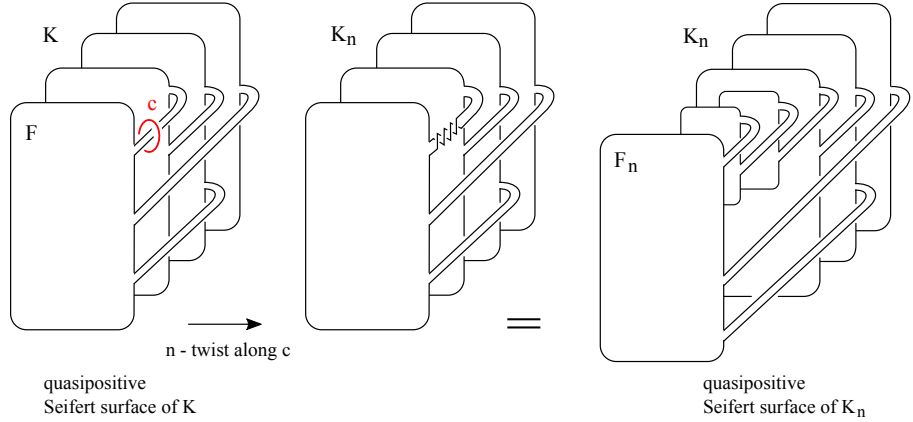}
\caption{$\wind_c(K) = 0$ and $K_n$ bounds a quasipositive Seifert surface for $n < 0$; here $n = -2$.}
\label{quasipositive}
\end{center}
\end{figure}

Theorem~\ref{thm:Murasugi-sum} enables us to further refine the above statement for tight fibered knots. 

\begin{theorem}
\label{thm:postwistquasipositive}
Let $\{ K_n \}$ be a twist family of knots obtained by twisting $K$ along $c$ such that $c$ is neither split from $K$ nor a meridian of $K$. 
If $K_n$ is a tight fibered knot for infinitely many integers $n$, 
then $c$ links $K$ coherently and there is a constant $N$ such that 
\begin{itemize}
\item $K_n$ is tight fibered for every $n \geq N$ and
\item $K_n$ is tight fibered for only finitely many $n < N$.
\end{itemize}
\end{theorem}

Pushing these ideas further, 
we obtain a characterization of when the twisting circle $c$ is actually a braid axis for the twist family $\{K_n
\}$.

\begin{theorem}
\label{thm:braidaxis}
Let $\{K_n\}$ be a twist family obtained by twisting a knot $K$ along an unknot $c$ that is neither split from $K$ nor a meridian of $K$.
Then $c$ is a braid axis of $K$ if and only if both $K_n$ and $\overline{K_{-n}}$ are tight fibered for sufficiently large $n$.
\end{theorem}

\subsection{L-space knots}\label{sec:introLspace}
Recall that an {\em L-space} is a closed, compact, connected, oriented rational homology sphere $3$--manifold 
with ``minimal Heegaard Floer homology'' \cite{OS_lens}.
In this paper, we say a knot $K$ in $S^3$ is an {\em L-space knot} if $r$--surgery on $K$ produces an L-space for some slope $r \neq \infty$ 
(so $r$ is not the meridional slope and necessarily not $0$).  
We will say $K$ is a {\em positive} or {\em negative} L-space knot according to the sign of $r$; 
only the unknot is both a positive and negative L-space knot. 

Since every positive L-space knot is a tight fibered knot (\cite[Proposition~2.1]{Hed_positive} with \cite{Ghi,Ni, Ni2, Juh}), 
Corollary~\ref{cor:sqptwist}(2) and Theorem~\ref{thm:postwistquasipositive} apply to positive L-space knots as well.
Furthermore, 
we can confirm the intuition that L-space knots resulting from sufficiently large positive twistings are indeed positive L-space knots.  
Theorem~\ref{thm:positivetwistLspaceknots} below answers \cite[Question~7.2]{BM} in the positive and also improves \cite[Corollary~1.9 and Theorem~1.11]{BM}. 
Furthermore Theorem~\ref{thm:positivetwistLspaceknots}, 
together with \cite[Theorem~1.4]{BM}, settles the Conjecture 1.3 in \cite{BM} in the positive. 

\begin{theorem}
\label{thm:positivetwistLspaceknots}
Let $\{ K_n \}$ be a twist family of knots obtained by twisting $K$ along $c$ such that $c$ is neither split from $K$ nor a meridian of $K$.  
If $K_n$ is an L-space knot for infinitely many integers $n>0$ \(resp.\ $n < 0$\), then 
\begin{itemize}
\item $c$ links $K$ coherently and
\item there is a constant $N$ such that $K_n$ is a positive \(resp. negative\) L-space knot for all integers $n \geq N$ \(resp. $n \leq N$\).
\end{itemize}
\end{theorem}

The second assertion of this theorem will follow from Theorem~\ref{thm:limit} below which relates L-space surgeries and limits of twist families. 
Given a twist family of knots $\{ K_n \}$ obtained by twisting a knot $K$ along an unknot $c$, 
we can also consider the limit knot $\displaystyle K_{\infty} = \lim_{n\to \infty} K_n$.  
To make this more precise, consider that $K_n$ is the image of $K$ upon $(-1/n)$--surgery on $c$.  
Since $\displaystyle 0 = \lim_{n\to \infty} (-1/n)$, 
we define $K_{\infty}$ to be the image of $K$ upon $0$--surgery on $c$.  
Hence the limit knot $K_{\infty}$ is a knot in $S^1 \times S^2$.    
A slope $r$ on $K$ similarly yields a slope $r_n$ on $K_n$ and a slope $r_{\infty}$ on $K_{\infty}$.  
In particular, the manifold $K_{\infty}(r_{\infty})$ which arises as the limit of the manifolds $K_n(r_n)$ is the manifold obtained by $(r,\ 0)$--surgery on $K \cup c$. 
If $r$ is not a meridional slope on $K$, 
then $r_{\infty}$ is not a meridional slope on $K_{\infty}$ neither.

We will also say a knot $K_{\infty} \subset S^1 \times S^2$ is an {\em L-space knot} if some Dehn surgery on $K_{\infty}$ produces an L-space.  
An L-space knot $K$ in $S^1 \times S^2$ enjoys the following remarkable properties. 
 
\begin{remark}
\label{L-space_knot_S1xS2}
Assume that $K_{\infty}$ is an L-space knot in $S^1 \times S^2$. 
Then 
\begin{enumerate}
	\item  the exterior of $K_{\infty}$ is a ``generalized solid torus'' so that every non-trivial surgery on $K_{\infty}$ is an L-space \cite{RR, Gillespie}, and
	\item  $K_{\infty}$ is a {\em spherical braid}, it may be isotoped in $S^1 \times S^2$ to be transverse to each of the $S^2$ fibers \cite{NiVafaee, RR}.
\end{enumerate}
\end{remark}

\begin{theorem}
\label{thm:limit}
Let $\{ K_n \}$ be a twist family of knots obtained by twisting $K$ along $c$.
\begin{enumerate}
\item If $K=K_0$ is an L-space knot with L-space surgery slope $r$ and $K_{\infty}$ is an L-space knot, then 
\begin{enumerate}
\item $K_n$ is a positive L-space knot for all $n \geq 0$ if $r > 0$ and
\item $K_n$ is a negative L-space knot for all $n \leq 0$ if $r < 0$.
\end{enumerate}
\item If $K_n$ is an L-space knot for infinitely many integers $n$, 
then $K_{\infty}$ is an L-space knot.
\end{enumerate}
\end{theorem}

Since the unknot is both a positive and a negative L-space knot, we obtain the following corollary from Theorem~\ref{thm:limit} and Theorem~\ref{thm:braidaxis}.  

\begin{corollary}
\label{cor:exchangebraided}
Let $K \cup c$ be a non-trivial link of two unknots.
Let $\{ K_n \}$ be the twist family of knots obtained by twisting $K$ along $c$, and let $\{ c_n \}$ be the twist family of knots obtained by twisting $c$ along $K$.
Then either $\{ K_n \}$ and $\{ c_n \}$ each contain only finitely many L-space knots, or
\begin{enumerate}
\item $c$ is a braid axis for $K$,
\item $K$ is a braid axis for $c$,
\item $K_n$ is an L-space knot for all integers $n$, and
\item $c_n$ is an L-space knot for all integers $n$
\end{enumerate}
\end{corollary}

\begin{proof}
Assume $\{K_n\}$, say, contains infinitely many L-space knots.  Then Theorem~\ref{thm:limit}(2) implies that $K_{\infty}$ is an L-space knot in $S^1 \times S^2$. Since $K=K_0$ is an unknot,  Theorem~\ref{thm:limit}(1) then implies that $K_n$ is an L-space knot for all integers $n$ (a positive L-space knot for $n\geq0$ and a negative L-space knot for $n\leq0$) from which  Theorem~\ref{thm:braidaxis} implies  $c$ is a braid axis for $K$.  

Because $K_{\infty}$ is an L-space knot in $S^1 \times S^2$ obtained as the image of $K$ after $0$--surgery on $c$, 
$(K \cup c)(0,0)$ is an L-space; see Remark~\ref{L-space_knot_S1xS2}(1). 
  Thus the knot $c_\infty$, obtained as the image of $c$ after $0$--surgery on $K$, is also an L-space knot in $S^1 \times S^2$.  Since $c=c_0$ is an unknot, applying Theorem~\ref{thm:limit}(1) to the twist family $\{c_n\}$ implies that $c_n$ is an L-space knot for all integers $n$ from which Theorem~\ref{thm:braidaxis} implies  $K$ is a braid axis for $c$. 
\end{proof}

A link of two unknots in which each component is a braid axis for the other is called {\em exchangeably braided} \cite{Morton}.  So Corollary~\ref{cor:exchangebraided} above shows that if a twist family $\{K_n\}$ obtained by twisting an unknot $K=K_0$ about an unknot $c$ contains infinitely many L-space knots, then $K \cup c$ is exchangeably braided. 
For a given integer $\ell > 0$, 
\cite[Corollary~1.2]{Morton} shows that 
there are only a finite number of exchangeably braided links with linking number $\ell$. 
Hence the unknot $K$ has only finitely many twisting circles $c$ (up to isotopy) with $\omega = \ell$ each of which provides 
a twist family $\{ K_n \}$ consisting of L-space knots.
Note that 
(1) for each $\ell \ge 3$, we can take such a twisting circle $c$ so that $\omega = \ell$ and $K \cup c$ is a hyperbolic link,  
and (2) the number of such circles $c$ can be made arbitrarily large for suitably large $\ell$.   
See \cite[Propositions~7.1 and 7.2]{Mote}. 

\begin{question}
\label{exchangeably braided}
If the link of unknots $K \cup c$ is exchangeably braided, then does the twist family $\{K_n\}$ contain infinitely many L-space knots?
\end{question}

\begin{remark}
\label{twisted unknot}
Having an unknot in a twist family is not a requirement for the family to consist of only L-space knots.  
Indeed, many twist families of torus knots contain no unknots. 
One may also readily construct twist families of $1$--bridge braids that do not contain an unknot, 
and these are all L-space knots \cite{GLV}.  
The second author has even constructed a twist family of (mutually distinct) hyperbolic L-space knots with tunnel number two \cite[Theorem 8.1(1)]{Mote}.
\end{remark}

\subsection{Unknotting number}

We develop Theorem~\ref{thm:slice_genus_bound} to gain control on the slice genus of knots in twist families for the proof of Theorem~\ref{limit_omega>=1}.
\begin{theorem}
	\label{thm:slice_genus_bound}
	Let $\{K_n\}$ be a twist family with winding number $\omega$. 
	For $n > 0$ there exist constants $C$ and $C'$ such that 
	\[ C + n\omega (\omega-1) \le 2g_4(K_n) \le C' + n\omega (\omega-1).\]
\end{theorem}
However, since the unknotting number $u(K)$ of a knot $K$ is bounded below by its slice genus $g_4(K)$, 
this immediately informs us about the growth of unknotting number for twist families of knots with $\omega \geq 2$.

\begin{corollary}
	\label{cor:unknotting}
	Let $\{ K_n \}$ be a twist family of knots obtained by twisting $K$ along $c$ with $\wind_c(K) \geq 2$. 
	Then $u(K_n) \to \infty$ as $n \to \infty$. 
\end{corollary}

\begin{proof}
	It follows from Theorem~\ref{thm:slice_genus_bound} that 
	$C + n \omega (\omega-1) \leq 2g_4(K_n)$ for some constant $C$. 
	Hence $g_4(K_n)$ tends to $\infty$ as $n \to \infty$. 
	Since the unknotting number $u(K)$ of a knot $K$ is bounded below by the slice genus $g_4(K)$, 
	we have the desired conclusion. 
\end{proof}

In contrast, there are twist families of knots with $\omega =0$ and others with $\omega =1$ for which the unknotting number is bounded.
	
\begin{example}	
For the twist family of knots $\{K_n\}$  obtained from the Whitehead link $K \cup c$, one readily sees that $u(K_0)=0$ while $u(K_n)=1$ otherwise.
\end{example}

\begin{example}
\label{wind1_unknotting1}
Let us consider a knot $K$ and a twisting circle as depicted on the left of Figure~\ref{twist_unknotting1} 
with $\wind_c(K) = 1$.  
Performing crossing change at the crossing indicated by $*$ on the right of the figure, 
we obtain the trivial knot for any integer $n$. 
Hence $u(K_n) \leq 1$ for all integers $n$.

\end{example}

\begin{figure}[h]
	\begin{center}
		\includegraphics[width=0.47\linewidth]{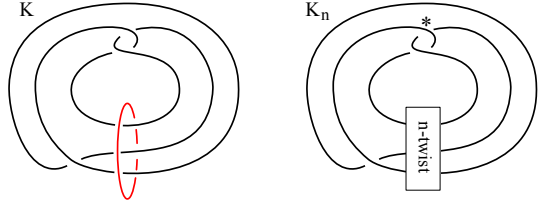}
		\caption{$K_n$ has unknotting number at one for all integers $n$.}
		\label{twist_unknotting1}
	\end{center}
\end{figure}

\subsection{Organization and Notation of the paper}
We assemble constraints on the Seifert and slice genera of knots in twist families and then prove Theorem~\ref{thm:slicegenera_seifertgenera} in Section~\ref{Seifert/slice}.  
In Section~\ref{coherent family} we refine our understanding of this for coherent twist families and prove Theorems~\ref{thm:coherent} and \ref{thm:Murasugi-sum}.  
Then we pull these together in Section~\ref{examples} to prove Theorem~\ref{thm:values}. 

In Section~\ref{tight_fibered_family} we examine twist families of tight fibered knots and build to a proof of Theorem~\ref{thm:braidaxis} 
that characterizes braid axes.   
Since L-space knots are tight fibered knots, this leads us to the discussion in Section~\ref{L-space_knot_family} of L-space knots in twist families and our contribution to the study of satellite L-space knots in Section~\ref{satellite}.

\medskip

Throughout the paper we will use $N(*)$ to denote a tubular neighborhood of $*$ and 
use $\nbhd(*)$ to denote the interior of $N(*)$ for notational simplicity. 
For a twist family $\{K_n\}$ obtained by twisting a knot $K$ along an unknot $c$, 
we will often use the abbreviations $\omega = \wind_c(K)=|lk(K,c)|$ for the winding number of $K$ about $c$ 
and $\eta = \wrap_c(K)$ for the wrapping number of $K$ about $c$.

\bigskip

\noindent
\textbf{Acknowledgments.}\quad 
KB would like to thank Jeremy Van-Horn-Morris for discussions that lead to Proposition~\ref{prop:upperslicebound} and Sarah Rasmussen for conversations about satellite L-space knots. 
KB and KM would like to thank Liam Watson for updates on the status his work with J.\ Hanselman and J.\ Rasmussen on Conjecture~\ref{L-slopes}.

\section{Seifert genus and slice genus of knots in twist families}
\label{Seifert/slice}

The goal in this section is to establish Theorem~\ref{thm:slicegenera_seifertgenera}. 
We will prove the theorem under slightly stronger conditions on the signs of $n$;
though, by considering the mirrored twist family $\{ \overline{K_n} \} = \{ \overline{K}_{-n}\}$ of $\{ K_n \}$, 
we obtain the result stated in Section~\ref{section:Introduction}. 

\begin{thm:slicegenera_seifertgenera} 
Let $\{ K_n \}$ be a twist family of knots obtained by twisting along $c$. 
Suppose that we have one of the following: 
\begin{enumerate}
\item 
There exists a constant $C$ such that $g(K_n)-g_4(K_n) \leq C$ for infinitely many integers $n > 0$. 
\item
$g(K_n) / g_4(K_n) \to 1$ as $n \to \infty$.
\end{enumerate}
Then either $\wind_c(K)= 0$ or $c$ links $K$ coherently. 
\end{thm:slicegenera_seifertgenera}

Among disks bounded by $c$ in $S^3$, 
let $\hat{D}$ be one for which $\eta = |K \cap \hat{D}|>0$ is minimized, 
i.e.\  $\eta$ is the wrapping number of $K$ about $c$ and the intersection number of $K$ with $\hat{D}$ realizes $\wrap_c(K)$. 
Observe that $\eta \geq \omega$ and $c$ bounds a disk that $K$ links coherently exactly when $\eta = \omega$.  
More explicitly, $K$ and $\hat{D}$ may be oriented so that $K$ has $\omega+d$ positive and $d$ negative intersections with $\hat{D}$ where $2d=\eta-\omega$.

Let $E = S^3-\nbhd(K \cup c)$ be the exterior of the link $K \cup c$ and set $D = \hat{D} \cap E$.  
Denote the Thurston norm on $H_2(E, \bdry E)$ by $x$ \cite{ThurstonNorm}. 
Since $\hat{D}$ is intersected $\eta$ times by $K$, $x([D]) \leq \eta - 1$. 
On the other hand, since $\partial [D] = [\partial D] = \omega \mu_K + \lambda_c$
(where $\mu_K$ is a meridian of $K$ and $\lambda_c$ is a preferred longitude of $c$), 
we have $\omega-1 \leq x([D])$. 
Thus $\omega-1 \leq x([D]) \leq \eta - 1$ unless $c$ is split from $K$ so that $D$ is an unpunctured disk.

\begin{theorem}[Theorem~5.3 \cite{BT}]
\label{thm:twistgenera}
There is a constant $G=G(K,c)$ such that $2g(K_n) = 2G + n \omega x([D])$ for sufficiently large $n >0$. 
\end{theorem}

Here we obtain a similar estimate on the slice genus of all knots in the twist family. 
Recall that $g_4(K)$ denotes the slice genus ($4$--ball genus) of a knot $K$. 
Similarly, for an oriented link $L$, 
let $g_4(L)$ denote the slice genus of $L$. 

\begin{thm:slice_genus_bound}
For $n > 0$ there exist constants $C$ and $C'$ such that 
\[ C + n\omega (\omega-1) \le 2g_4(K_n) \le C' + n\omega (\omega-1).\]
\end{thm:slice_genus_bound}

To prove Theorem~\ref{thm:slice_genus_bound} 
we will prove Propositions~\ref{prop:upperslicebound} and \ref{prop:newlowerslicebound} below,  
which give an upper bound and 
a lower bound of $g_4(K_n)$, respectively.

\begin{proposition}
\label{prop:upperslicebound}
For each pair of integers $m, n$, 
 \[2|g_4(K_n)-g_4(K_m)| \leq  |n-m| \omega(\omega-1) + (\eta-\omega).\]  
In particular, setting $2G_4 = 2g_4(K_0)+ \eta-\omega$, 
we have 
\[ 2g_4(K_n) \leq 2G_4 + |n| \omega(\omega-1).\]
\end{proposition}

\begin{proof} 
Choose an orientation on $K$ and thus on the knots $K_n$.  
View $K_n$ as the result of $(n - m)$ twists of $K_m$ about $c$.
Then a sequence of $\eta$ oriented bandings separates $K_n$ into the split link $K_m \sqcup T_{\eta,\ (n-m) \eta}$ as indicated in Figure~\ref{cobordism_twist}.  
Since each oriented banding may be viewed as giving a saddle surface in $S^3 \times [0, 1]$, 
this shows that $K_n$ is cobordant to $K_m \sqcup T_{\eta,\ (n-m) \eta}$ by a planar surface $P$ with $\eta+2$ boundary components. 

The torus link $T_{\eta,\ (n-m) \eta}$ has the orientation induced from $K_n$ by the bandings; namely, 
$\omega+d$ components run in one direction while $d$ components run in the other, 
where $2d = \eta-\omega$.   
Thus there are $d$ pairwise disjoint annuli each of which is cobounded by a pair of components of 
$T_{\eta,\ (n-m) \eta}$ with opposite orientations.  
Take a band on each annulus which connects boundary components, 
and apply a sequence of oriented bandings to separate $T_{\eta,\ (n-m) \eta}$ into the split link 
$T_{\omega,\ (n-m) \omega} \sqcup [d\textrm{--component trivial link}]$. 
Each oriented banding corresponds to a saddle surface in $S^3 \times [0, 1]$, 
which gives a cobordism from $T_{\eta,\ (n-m) \eta}$ to the split union of the coherently oriented torus link $T_{\omega,\ (n-m) \omega}$ and $d$--component trivial link. 
Then capping off the $d$--component trivial link by disks, 
we obtain a cobordism from $T_{\eta,\ (n-m) \eta}$ to the coherently oriented torus link $T_{\omega,\ (n-m) \omega}$. 
See Figure~\ref{cobordism_twist}. 
Let $F_{\omega,\ (n-m) \omega}$ be the fiber surface of $T_{\omega,\ (n-m) \omega}$. 
Then following  \cite[\S4.2]{Cavallo} we have: 

\[g_4(T_{\omega,\ (n-m) \omega}) 
= g(F_{\omega,\ (n-m) \omega}) 
= \frac{(\omega-1)(|n-m| \omega -1) + 1 - \omega}{2} 
= \frac{|n-m|\omega(\omega-1)}{2} + 1-\omega. \]

Also let $S_m$ be a slice surface for $K_m$ such that $g(S_m) = g_4(K_m)$.  
Then we construct a surface $S_{n,\,m}$ bounded by $K_n$ by attaching 
$S_m$,\ $F_{\omega,\ (n-m) \omega}$, 
and $d$ annuli to $P$ as in Figure~\ref{cobordism_twist}.  
Since $F_{\omega,\ (n-m) \omega}$ is a connected surface with $\omega$ boundary components, 
one readily calculates that 
\begin{align*}
g(S_{n,\, m}) &= g(S_m) + \left( g(F_{\omega,\ (n-m)\omega}) + (\omega-1) \right) + d \\
& = g_4(K_m) + \frac{|n-m|\omega(\omega-1)}{2} + d. 
\end{align*}  
\begin{figure}
\begin{center}
\includegraphics[width=1.0\textwidth]{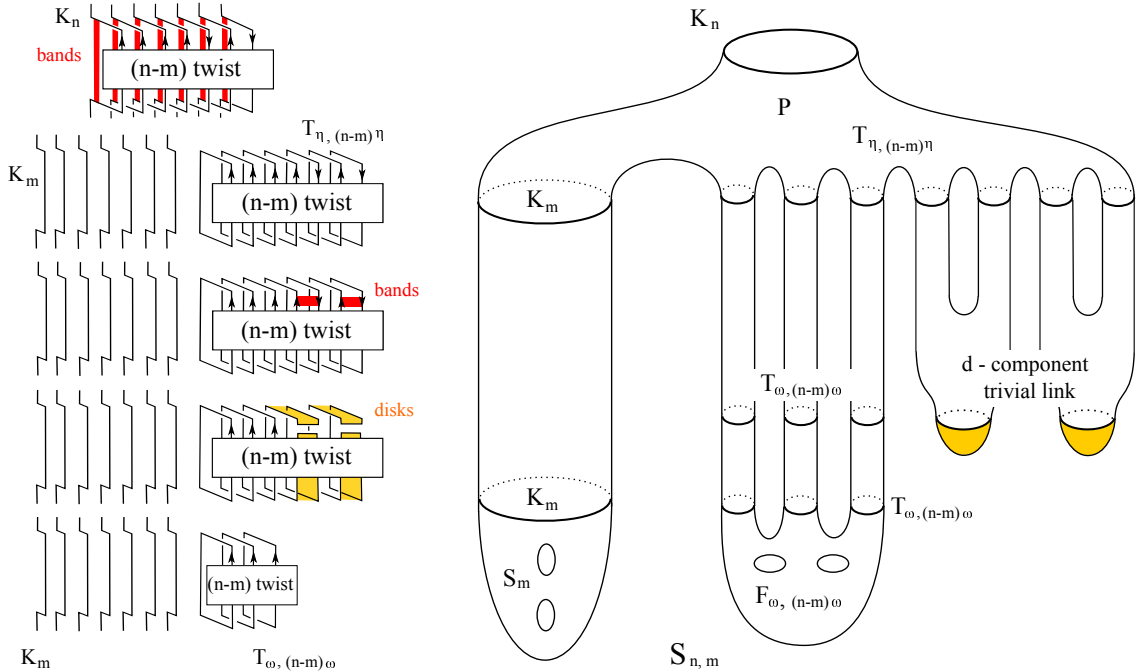}
\caption{$\omega = 3, \eta = 7, d = 2$}
\label{cobordism_twist}
\end{center}
\end{figure}
So then  
\[
2g_4(K_n) \leq 2g(S_{n,\, m}) = 2g_4(K_m) + |n-m|\omega(\omega-1) + 2d.  
\] 
Viewing $K_m$ as the result of $m-n$ twists on $K_n$ about $c$, 
the above argument also shows that 
\[
2g_4(K_m) \leq 2g_4(K_n) + |m-n|\omega(\omega-1) + 2d.  
\] 
The claimed result follows from these two inequalities. 
\end{proof}

\begin{question}
\label{question_w=0}
By Proposition~\ref{prop:upperslicebound}, 
for a twist family $\{K_n\}$ with $\omega \le 1$, 
the sequence $\{g_4(K_n)\}$ is bounded. 
Must $g_4(K_n)$ be constant except for finitely many $n\in\Z$?
\end{question}

Using the Rasmussen invariant of knots  \cite{Ras}, 
we now obtain a lower bound on $g_4(K_n)$ analogous to the upper bound obtained in Proposition~\ref{prop:upperslicebound}.
Recalling that Beliakova and Wehrli \cite{BW} extended the Rasmussen invariant of knots \cite{Ras} to oriented links, 
let $s(L)$ denote the Rasmussen invariant of an oriented link $L$. 

\begin{proposition}
\label{prop:newlowerslicebound}
Let $\{K_n\}$ be a twist family with winding number $\omega$ and wrapping number $\eta$.
For integers $n \geq m$,
\[ (n-m) \omega (\omega-1) - 2\eta \leq s(K_n) - s(K_m).\]
Consequently, for $n \geq 0$
 \[ s(K_0)  - 2\eta + n \omega (\omega-1) \leq 2g_4(K_n).\]
\end{proposition}

\begin{proof}
Set $d$ so that $2d = \eta-\omega$.
As constructed in Proposition~\ref{prop:upperslicebound} there is a connected planar surface $P$ giving a cobordism from 
$K_n$ to $K_m \sqcup T_{\eta,\ (n-m) \eta}$ 
where $T_{\eta,\ (n-m) \eta}$ has $d$ components running in one direction and $\omega + d$ running in the other. 
Since $|\bdry P| = \eta+2$, $-\chi(P) = \eta$.
Therefore \cite[(7.1)]{BW} gives
\[ |s(K_m \sqcup T_{\eta,\ (n-m)\eta}) - s(K_n)| \leq \eta.\]
Since $s(L_1 \sqcup L_2) = s(L_1) + s(L_2) -1$ by \cite[(7.2)]{BW}, we obtain
\[ | s(K_m) - s(K_n) +s(T_{\eta,\ (n-m)\eta}) - 1 | \leq \eta \]
and hence 
\[s(T_{\eta,\ (n-m)\eta}) - 1-\eta \leq s(K_n)-s(K_m). \tag{$\star$} \]
Now $d$ bandings of pairs of anti-parallel components of $T_{\eta,\ (n-m) \eta}$ produces a weak cobordism $A$ 
(composed of straight annuli and $d$ pairs of pants) to the split link $T_{\omega,\ (n-m) \omega} \sqcup U_d$ with $-\chi(A) = d$.  
(Here $U_d$ is the $d$--component trivial link.) 
Hence, using  \cite[(7.1) \& (7.2)]{BW} again along with the computations $s(U_d) = 1-d$ of \cite[\S 7.1]{BW} and
$s(T_{\omega,\ (n-m)\omega}) = (\omega-1)((n-m)\omega-1)$ of \cite[\S4.2]{Cavallo}, 
we obtain
 \begin{align*}
| s(T_{\omega,\ (n-m)\omega} \sqcup U_d) - s(T_{\eta,\ (n-m)\eta}) | & \leq d \\
| s(T_{\omega,\ (n-m)\omega}) + s(U_d)-1 - s(T_{\eta,\ (n-m)\eta}) | & \leq d \\
| (\omega-1)\bigl((n-m)\omega-1\bigr) + (1-d) - 1 - s(T_{\eta,\ (n-m)\eta}) | & \leq d
\end{align*}
so that
\begin{align*}
(\omega-1)\bigl((n-m)\omega-1\bigr) - 2d &\leq s(T_{\eta,\ (n-m)\eta})\\
(n-m)\omega(\omega-1)+1 - \eta &\leq s(T_{\eta,\ (n-m)\eta}).
\end{align*}
Therefore, with ($\star$) we obtain
 \[(n-m)\omega(\omega-1) -2\eta \leq s(K_n)-s(K_m)\]
as claimed. 
Put $m = 0$ and use $s(K_n) \leq 2g_4(K_n)$ \cite[Theorem~1]{Ras} 
to obtain the second conclusion.
\end{proof}

\medskip

\begin{proof}[Proof of Theorem~\ref{thm:slice_genus_bound}]
When $n > 0$, 
Propositions~\ref{prop:upperslicebound} and \ref{prop:newlowerslicebound} give
\[ s(K_0)  - 2\eta + n\omega (\omega-1) \leq 2g_4(K_n) \leq 2G_4 + n \omega(\omega-1)\]
Putting $C = s(K_0)  - 2\eta$ and $C' = 2G_4$, 
we obtain the desired result. 
\end{proof}

Combining Theorem~\ref{thm:twistgenera} and \ref{thm:slice_genus_bound}, 
we obtain: 

\begin{theorem}
\label{limit_omega>=1}
Suppose that $\omega = \wind_c(K) \ge 1$. 
Then we have:
\begin{enumerate}
\item
If $\omega = 1$ but $\eta>1$, 
then $\displaystyle \lim_{n\to \infty} \frac{g(K_n)}{g_4(K_n)} = \infty$. 
\item
If $\omega > 1$, 
then $\displaystyle \lim_{n\to \infty} \frac{g(K_n)}{g_4(K_n)} =  \frac{x([D])}{\omega-1}$.
\end{enumerate}
\end{theorem}

\begin{proof}
(1) 
Under the assumption Theorem~\ref{thm:twistgenera} 
shows $g(K_n) \to \infty$ as $n \to \infty$,  
while Proposition~\ref{prop:upperslicebound} shows that $g_4(K_n)$ is bounded above. 
This gives the desired result. 

(2) For suitably large integers $n$, 
Theorem~\ref{thm:twistgenera} shows that $2g(K_n) = 2G + n \omega x([D])$ for some constant $G$.
When $n > 0$, 
Theorem~\ref{thm:slice_genus_bound} gives
\[ C + n\omega (\omega-1) \leq 2g_4(K_n) \leq C' + n \omega(\omega-1)\]
where the left hand side is positive for suitably large $n$.   
Hence, for such large $n$, 
we have
\[\frac{2G + n \omega x([D])}{C' + n \omega(\omega-1)} \leq \frac{g(K_n)}{g_4(K_n)} \leq  \frac{2G + n \omega x([D])}{C + n\omega (\omega-1)}.\]
Assuming $\omega >1$ and taking the limit as $n \to \infty$, 
we then obtain
\[\frac{x([D])}{\omega-1} \leq \lim_{n\to \infty} \frac{g(K_n)}{g_4(K_n)} \leq  \frac{x([D])}{\omega-1}\]
 which yields the desired result.
\end{proof}

Now we are ready to establish Theorem~\ref{thm:slicegenera_seifertgenera}. 

\noindent
\begin{proof}[Proof of Theorem~\ref{thm:slicegenera_seifertgenera}]
We assume that $\omega = \wind_c(K) \geq 1$ and aim to show that 
$c$ links $K$ coherently under each of the hypotheses (1) and (2).  
Let $S$ be a properly embedded surface in $E = S^3 - \nbhd(K \cup c)$ with coherently oriented boundary that both represents $[D]$ 
and realizes $x([D])$.  (Recall $D = \hat{D}\cap E$.)
Since $[S] = [D] \in H_2(E, \partial E)$, 
$\partial [S] = \partial [D] \in H_1(\partial E)$ and 
$[\partial S] = [\partial D] = \omega \mu_K + \lambda_c$, where $\mu_K$ is a meridian of $K$ and $\lambda_c$ is a preferred longitude of $c$. 
Hence $\bdry S$ consists of $\omega$ curves in $\partial N(K)$ representing meridians of $K$ 
and $1$ curve in $\partial N(c)$ representing the preferred longitude of $c$.   
Consequently, $x([D]) = 2g(S) + \omega - 1$ and if $g(S) = 0$ then $S$ caps off to a disk $\hat{S}$ in $S^3$ bounded by $c$ that $K$ intersects coherently.   
So in the following we show that $g(S) =0$ follows from each of the hypotheses (1) and (2).

First, recall from Theorem~\ref{thm:twistgenera} and Proposition~\ref{prop:upperslicebound} that 
\begin{itemize}
\item
$2g(K_n) = 2G + n\omega x([D]) = 2G  + 2n\omega g(S) + n\omega(\omega-1)$\ for sufficiently large $n$, and 
\smallskip
\item
$2g_4(K_n) \le 2G_4 + |n|\omega(\omega-1)$ where we set $G_4 = g_4(K_0) + \frac{(\eta-\omega)}{2}$.
\end{itemize}
Note that we may assume $G_4>0$: if $G_4=0$ then $\eta-\omega=0$ and thus $c$ links $K$ coherently already.
\medskip

{\bf Case (1).}  
Assume there is a constant $C$ such that $g(K_n)-g_4(K_n) \leq C$ for infinitely many integers $n>0$.

For each sufficiently large $n>0$ such that $g(K_n) - g_4(K_n) \leq C$, 
we have  
\[G - G_4 + ng(S) \leq G - G_4 + n \omega g(S) \leq g(K_n) - g_4(K_n) \leq C\] 
where the first inequality is due to the assumption that $\omega \geq 1$.   
Since neither $C$, $G_4$, nor $G$ depend on $n$, 
the inequality $G - G_4 + ng(S) \leq C$ can only hold true for infinitely many positive integers $n$ if $g(S) = 0$.  

\medskip

{\bf Case (2).}  
Assume $g(K_n)/g_4(K_n) \to 1$ as $n \to \infty$. 

If $\omega = 1$, 
then Theorem~\ref{limit_omega>=1}(1) shows $\eta = 1$ or otherwise 
$\displaystyle \lim_{n\to \infty} \frac{g(K_n)}{g_4(K_n)} = \infty$. 
Hence $\eta = 1$, and $c$ links $K$ coherently exactly once. 

So we may now assume $\omega > 1$.  
Then by Theorem~\ref{limit_omega>=1}(2) 
\[ 1=\lim_{n\to \infty} \dfrac{g(K_n)}{g_4(K_n)} 
= \frac{x([D])}{\omega-1} 
= \frac{2g(S) + \omega -1}{\omega -1}
= \frac{2g(S)}{\omega-1} + 1.\]
Thus we must have $g(S) = 0$.
\end{proof}

\section{Coherent twist families of knots}
\label{coherent family}

Recall that a twist family is coherent if $\omega = \eta$; 
that is, $c$ bounds a disk which $K$ always intersects with the same orientation.   
The slice genus and the Seifert genus of a coherent twist family behave in a similar fashion asymptotically. 

\begin{thm:coherent}
Let $\{ K_n \}$ be a twist family of knots obtained by twisting $K$ along $c$. 
	
If $c$ links $K$ coherently at least twice, then both
\begin{enumerate}
\item there exists constants $C_+,C_- \geq 0$ such that $g(K_n)-g_4(K_n) = C_\pm$ for sufficiently large integers $\pm n$, 
and hence
\item $g(K_n) / g_4(K_n) \to 1$ as $|n| \to \infty$.
\end{enumerate}
\end{thm:coherent}

To prove this theorem, we need the following lemma. 
\begin{lemma}
\label{lem:sliceequality}
Let $\{K_n\}$ be a twist family in which $c$ links $K$ coherently. 
Then 
\[g_4(K_{n+1}) -g_4(K_n) = \frac{\omega(\omega-1)}{2} \]
for sufficiently large integers $n$. 
\end{lemma}

\begin{proof}
Since the twist family is coherent, Proposition~\ref{prop:upperslicebound} shows that 
\[g_4(K_{n+1}) -g_4(K_n) \leq \frac{\omega(\omega-1)}{2}  \tag{$*$}\] 
for all integers $n$. 
Observe that if ($*$) is a strict inequality for $k$ of the integers $n$ between $0$ and a positive integer $N$, 
then we have that 
\[g_4(K_N) \leq g_4(K_0) + \frac{N \omega(\omega-1)}{2} - k.\]
However, Proposition~\ref{prop:newlowerslicebound} 
(with $\eta = \omega$) shows that we must also have 
\[\frac{s(K_0)}{2} + \frac{N \omega(\omega-1)}{2} - \omega \leq g_4(K_N).\]
Together these two inequalities imply that 
\[k \leq g_4(K_0)-\frac{s(K_0)}{2} + \omega,\]
a bound independent of $N$. 
Hence there may be only at most $g_4(K_0)-\frac{s(K_0)}{2} + \omega$ integers for which ($*$) is a strict inequality.  
Thus ($*$) is an equality for sufficiently large $n$ and our conclusion holds.
\end{proof}

\begin{proof}[Proof of Theorem~\ref{thm:coherent}]
Since $c$ links $K$ coherently, we have $x([D]) = \omega - 1$ in Theorem~\ref{thm:twistgenera} so that for some constant $G$
\[g(K_n) = G+\frac{n\omega(\omega-1)}{2}\] 
for sufficiently large $n$.  
On the other hand, 
Lemma~\ref{lem:sliceequality} shows that 
\[g_4(K_{n+1}) -g_4(K_n) = \frac{\omega(\omega-1)}{2}\] 
for sufficiently large $n$.  
Therefore there is an integer $n_0$ such that for all $n \ge n_0$, 
\[g_4(K_n) 
= g_4(K_{n_0}) + \frac{(n-n_0)\omega(\omega-1)}{2} 
= g_4(K_{n_0}) - \frac{n_0\omega(\omega-1)}{2} + \frac{n\omega(\omega-1)}{2} 
= G' + \frac{n\omega(\omega-1)}{2},\]  
where $G'$ is the constant $g_4(K_{n_0}) - \frac{n_0\omega(\omega-1)}{2}$.  
Hence 
$g(K_n) - g_4(K_n) = G - G'$ for sufficiently large $n$.  
Putting  $C_+ = G - G'$ then gives half of the desired conclusion (1). 

Since the same argument applies to the mirrored twist family $\{\overline{K_n}\}$ where $\overline{K_n} = \overline{K}_{-n}$, 
we similarly obtain the constant $C_-$ and the other half of conclusion (1). 

It remains to see that $g(K_n) / g_4(K_n) \to 1$ as $|n| \to \infty$.  
Since $c$ links $K$ coherently, 
$\wind_c(K) \geq 2$ and $g(K_n) \to \infty$ as $|n| \to \infty$; 
see \cite[Theorem~2.1]{BM} or Theorem~\ref{thm:twistgenera}. 
Hence $g_4(K_n) = g(K_n) - C_\pm$ also tends to $\infty$ when $\pm n \to \infty$. 
Thus $\frac{g(K_n)}{g_4(K_n)} = \frac{g_4(K_n) + C_\pm}{g_4(K_n)} = 1 + \frac{C_\pm}{g_4(K_n)} \to 1$ as $\pm n \to \infty$ 
from which conclusion (2) follows. 
\end{proof}

In general there are integral sequences $a_n  \geq b_n$ for which $a_n/b_n \to 1$ as $n \to \infty$ and yet $a_n-b_n \to \infty$.
Corollary~\ref{cor:2=>1} below shows that this is not the case for the Seifert and slice genera of knots in a twist family.

\begin{corollary}
\label{cor:2=>1}
If $g(K_n)/g_4(K_n) \to 1$ as $n \to \infty$, 
then there is a constant $C$ such that $g(K_n) \le g_4(K_n) + C$ for all $n \ge 0$.  
\end{corollary}

\begin{proof}
By the assumption Theorem~\ref{thm:slicegenera_seifertgenera} shows that $\wind_c(K) = 0$ or $c$ links $K$ coherently. 
In the former case, $g_4(K_n) \le g(K_n) \le C$ for some constant $C$. 
Thus we have  the conclusion. 
So suppose that the latter case happens.
If $\eta = 1$, then $c$ is a meridian of $K$ and 
$K_n = K$ for all integers $n$, 
hence $g(K_n) - g_4(K_n) = g(K) - g_4(K)$, 
a constant $C$. 
If $\eta \ge 2$, 
then Theorem~\ref{thm:coherent} gives the desired result.
\end{proof}

\medskip

Furthermore, 
when $c$ links $K$ coherently, 
we can describe a minimal genus Seifert surface of $K_n$ explicitly for sufficiently large $n$ as follows.

\begin{thm:Murasugi-sum}
Let $\{K_n\}$ be a twist family in which $c$ links $K$ coherently $\omega>0$ times. 
Then there is an integer $n_0$ such that for every integer $n > n_0$, the knot
$K_n$ has a minimal genus Seifert surface that may be obtained as a Murasugi sum of a minimal genus Seifert surface for $K_{n_0}$ 
and the fiber surface for the $(\omega, (n-n_0) \omega + 1)$--torus knot. 
\end{thm:Murasugi-sum}

Before proving this theorem, 
for convenience, 
we introduce a variant of Murasugi sums. 
In Figure~\ref{fig:Murasugi_sum_variant-braid}, 
on the left is the usual Murasugi sum along a $2n$--gon (shown with $n=3$) together with a distorted version.  
On the right is a variant of the Murasugi sum and a similarly distorted version where the vertices of the $2n$--gon are expanded to arcs.  The two versions of the Murasugi sum produce isotopic results. 
(These distorted versions are shown to indicate how the variant is used in Figure~\ref{fig:twisting_Murasugi_sum} after further isotopy.) 

\begin{figure}[h]
\begin{center}
\includegraphics[width=\textwidth]{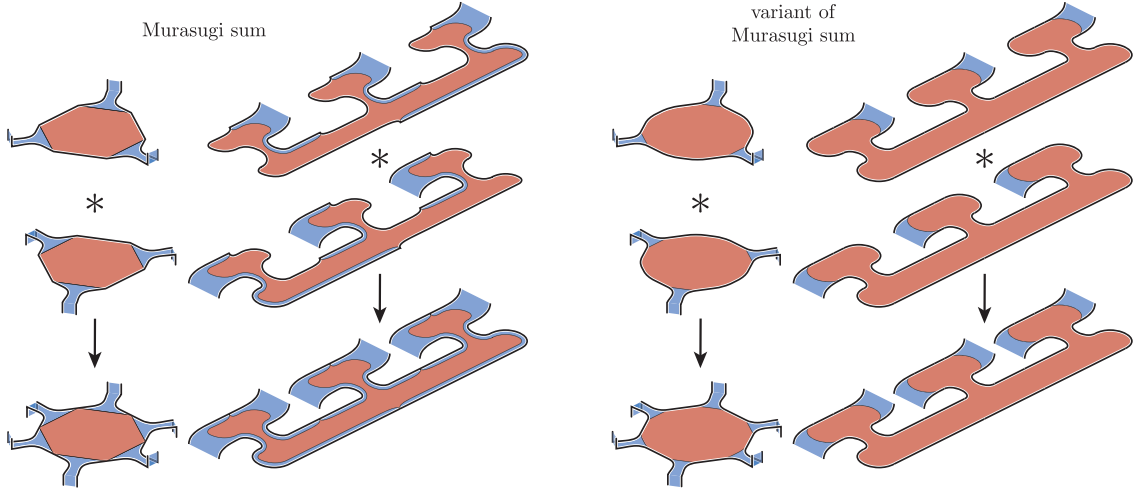}
\caption{Left: the standard Murasugi sum along a $2n$--gon.  Right: a variant of the Murasugi sum where the vertices of the $2n$--gon have been expanded to arcs in the boundary of the surfaces.} 
\label{fig:Murasugi_sum_variant-braid}
\end{center}
\end{figure}

\begin{proof}
First observe that if $K$ were a torus knot (i.e.\ a $0$--bridge braid) in the solid torus complement of $c$ or just the core of that solid torus, 
then the theorem is satisfied.  
Hence from here on we assume the exterior of $K \cup c$ is neither a product of a torus and an interval nor a cable space.

Let $\hat{D}$ be a disk that $c$ bounds and $K$ intersects coherently $\omega >0$ times, 
and let $D$ be the restriction of $\hat{D}$ to the exterior of $K \cup c$.  
Because $K$ intersects $\hat{D}$ coherently, we have that $x([D])=-\chi(D) = \omega-1$; 
see the proof of Theorem~\ref{thm:slicegenera_seifertgenera}. 
It then follows from Theorem~\ref{thm:twistgenera} that there  are constants $N_0$ and $G$ such that 
$2g(K_{N_0+k}) = 2G + (N_0 + k) \omega  (\omega-1)$ for every integer $k \ge  0$.
In particular, 
this implies the equality
\[g(K_{N_0 + k}) = g(K_{N_0}) + \frac{k\omega(\omega-1)}{2} \tag{$\dagger$} \] 
for every integer $k\geq0$.


Observe that for each integer $n$, 
the restriction of any Seifert surface for $K_n$ to the exterior of $K \cup c$ meets $\bdry N(c)$ in curves of slope $-1/n$.
Therefore \cite[Theorem 4.6]{BT} implies we may choose $N_0$ large enough 
so that $K_{N_0}$ has a minimal genus Seifert surface $\hat{F}_{N_0}$ that $c$ intersects coherently $\omega$ times.  
We may then isotope the Seifert surface $\hat{F}_{N_0}$ around the twisting circle $c$ as in Figure~\ref{fig:band_twist}.

Let $\hat{F}_{N_0 + 1}$ be a Seifert surface of $K_{N_0 + 1}$ obtained from $\hat{F}_{N_0}$ by adding $\omega(\omega-1)$ 
twisted bands as shown in Figure~\ref{fig:band_twist} (where $\omega = 3$), 
which is isotoped to a surface given in the rightmost picture of Figure~\ref{fig:band_twist}.
Then the equality $\chi(\hat{F}_{N_0 + 1}) = \chi(\hat{F}_{N_0}) - \omega(\omega-1)$ 
and the equality ($\dagger$) with $k = 1$ imply
\[g(\hat{F}_{N_0 + 1}) = g(\hat{F}_{N_0}) + \frac{\omega(\omega-1)}{2}  
= g(K_{N_0}) + \frac{\omega(\omega-1)}{2} = g(K_{N_0 + 1}). \]
This means that $\hat{F}_{N_0 + 1}$ is a minimal genus Seifert surface of $K_{N_0 + 1}$. 

\begin{figure}[h]
\begin{center}
\includegraphics[width=0.5\textwidth]{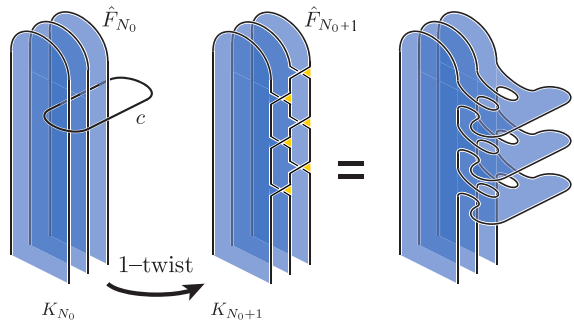}
\caption{The surface $\hat{F}_{N_0 + 1}$ is obtained from $\hat{F}_{N_0}$ by adding $\omega(\omega-1)$ twisted bands; shown here with $\omega = 3$.} 
\label{fig:band_twist}
\end{center}
\end{figure}

Now we follow the argument in the proof of \cite[Lemma~3.4]{BEVhM}.  
We put $n_0 = N_0 + 1$ and prove that 
if $n > n_0$, 
then a minimal genus Seifert surface of $K_n$ is obtained as a Murasugi sum of a minimal genus Seifert surface for $K_{n_0}$ and the fiber surface for the $(\omega, (n-n_0) \omega + 1)$--torus knot. 

\begin{figure}[h]
\begin{center}
\includegraphics[width=1.0\textwidth]{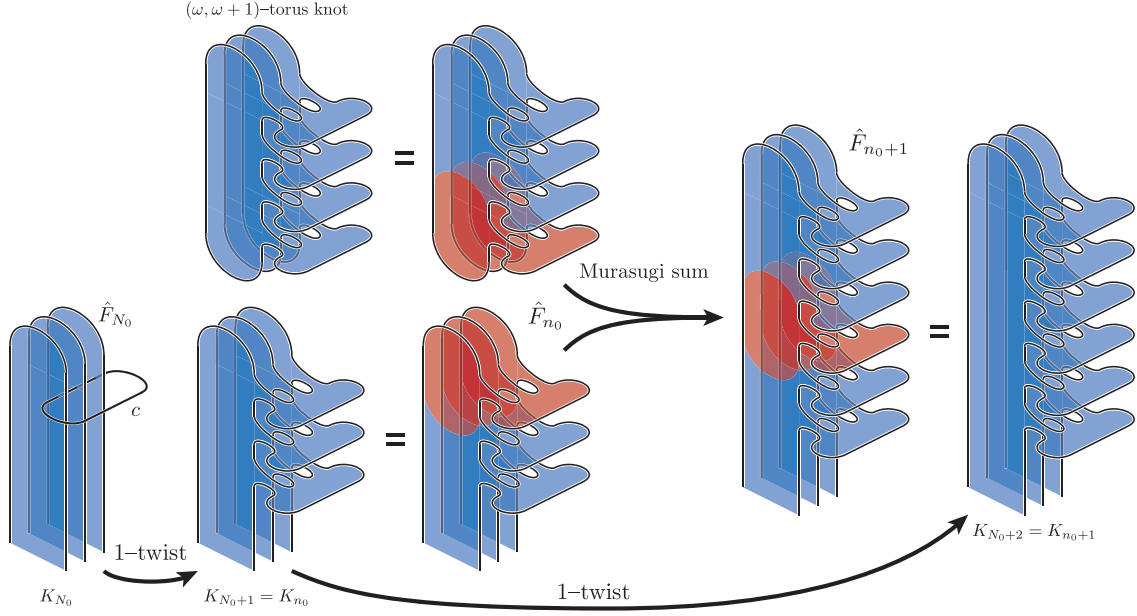}
\caption{A Murasugi sum with a fiber of a $(\omega,\omega+1)$--torus knot produces a twist; shown here with $\omega = 3$.}
\label{fig:twisting_Murasugi_sum}
\end{center}
\end{figure}

Let us take the two shaded disks; 
one is in $\hat{F}_{n_0} = \hat{F}_{N_0 + 1}$ 
and the other is in the minimal genus Seifert surface (the fiber surface) of the 
$(\omega, (n-n_0)\omega + 1)$--torus knot (as shown on the middle picture of Figure~\ref{fig:twisting_Murasugi_sum}). 
Then apply the Murasugi sum along these two disks as in Figure~\ref{fig:twisting_Murasugi_sum} to obtain 
$\hat{F}_{n}$ and $K_{n}$ as shown in Figure~\ref{fig:twisting_Murasugi_sum}, 
where $\omega = 3$ and $n = n_0 + 1$. 
Then recalling 
$g(\hat{F}_{n_0}) = g(K_{n_0})$ and $g(K_n) = g(K_{n_0}) + \frac{(n-n_0)\omega(\omega-1)}{2}$ 
(because $n_0 = N_0 + 1 > N_0$ and $n > n_0$), 
we have the equality 
\[
g(\hat{F}_{n})  = g(\hat{F}_{n_0}) + \frac{(n-n_0)\omega(\omega-1)}{2}
= g(K_{n_0}) + \frac{(n-n_0)\omega(\omega-1)}{2} 
= g(K_n). 
\]
This means that $\hat{F}_{n}$ is a minimal genus Seifert surface of $K_{n}$ and 
completes a proof. 
\end{proof}


\medskip

\section{Values of the limit of $g(K_n) / g_4(K_n)$}
\label{examples}

The purpose in this section is to understand the possible values of 
$\lim_{n \to \infty} \frac{g( K_n )}{g_4(K_n )}$ for a twist family $\{K_n\}$. When $\omega>0$, 
Theorem~\ref{thm:values} shows the limit exists and determines the range of values.  
When $\omega=0$, 
Proposition~\ref{prop:w=0} demonstrates that this limit may be any positive number that is at least $1$ or $\infty$ 
and Question~\ref{question_w=0} asks if this limit always defined.

\begin{thm:values}\
\begin{enumerate}
\item
For any twist family $\{K_n\}$ with $\omega > 0$, 
$\displaystyle \lim_{n \to \infty} \frac{g( K_n )}{g_4(K_n )} = r$ where $1 \le r \in \mathbb{Q} \cup \{\infty\}$.
\item
For any $1 \le r \in \mathbb{Q} \cup \{\infty\}$, 
there exists a twist family $\{ K_n \}$ with $\omega>0$ which satisfies 
$\displaystyle \lim_{n \to \infty} \frac{g( K_n )}{g_4(K_n )} = r$. 
\end{enumerate}
\end{thm:values}

\begin{proof}[Proof of Theorem~\ref{thm:values}]
First assume $\omega > 0$. 
It follows from Theorem~\ref{limit_omega>=1}
that $g(K_n)/g_4(K_n)$ converges to a rational number $r = x([D])/(\omega -1)$ when $\omega > 1$, 
otherwise it tends to $\infty$. 
Since $g(K_n) \ge g_4(K_n)$ for all integers $n$ for any twist family $\{K_n\}$, 
$r \ge 1$. 
This establishes the first assertion.
The second assertion follows from Theorem~\ref{thm:coherent} when $r = 1$, 
Theorem~\ref{limit_omega>=1}(1) when $r = \infty$, 
and Theorem~\ref{thm:proportion} below when $1 < r < \infty$.
\end{proof}

For a twist family with $\omega = \wind_c(K) >1$, 
observe that we necessarily have $x([D]) \geq \omega-1$ (where $D$ is the  disk bounded by $c$ and punctured by $K$).  
In Theorem~\ref{thm:proportion} we demonstrate that  $x([D])$ and $\omega-1$ are otherwise independent, 
that $\frac{x([D])}{\omega-1}$ may be any rational number greater than or equal to $1$. 

\begin{theorem}
\label{thm:proportion}
For any $1 <  r \in \mathbb{Q}$, 
there exists a twist family of knots $\{ K_n \}$ with $\omega>0$ which satisfies 
$\displaystyle \lim_{n \to \infty} \frac{g( K_n )}{g_4(K_n )} = r$. 
\end{theorem}

To prove this theorem we first need a technical lemma.

\begin{lemma}
\label{lem:boundedslicelimit}
If a twist family of knots obtained by twisting a knot $K$ along an unknot $c$ with winding number $\omega$ has bounded slice genus, 
then the twist family $\{(K^{p,q})_n\}$ obtained by twisting the $(p,q)$--cable of $K$ along $c$ satisfies
\[ \lim_{n \to \infty} \frac{g( (K^{p,q})_n )}{g_4(  (K^{p,q})_n )}  =1+ \frac{p}{p-1} \lim_{n \to \infty} \frac{2g(K_n)}{p \omega^2 n-1}.\]
\end{lemma}

\begin{proof}
For positive coprime integers $p,q$, 
let $K^{p,q}$ be the $(p,q)$--cable of $K$, 
the simple closed curve in $\bdry N(K)$ of slope $q/p$ and homologous to $p \lambda + q \mu$ 
where $\mu, \lambda$ are the standard meridian and longitude of $K$.  
By \cite[\S21 Satz 1]{Schubert}, $g(K^{p,q}) = p \, g(K) + (p-1)(q-1)/2$. 

A minimal genus Seifert surface for $K^{p,q}$ may be constructed as follows.
Push $K^{p,q}$ into the solid torus $N(K)$ and view it as the $(p,q)$--torus knot that wraps $p$ times.  As a torus knot in $N(K)$, the complement of $K^{p,q}$ has a fibration over $S^1$ with a fiber $S_{p,q}$ of genus $(p-1)(q-1)/2$ whose boundary consists of a longitude of $K^{p,q}$ and $p$ preferred longitudes of $K$.  So we may view $S_{p,q}$ as a ``relative'' Seifert surface of $K^{p,q}$ in $N(K)$.   Attaching $p$ copies of a minimal genus Seifert surface $F_K$ of $K$ to $S_{p,q}$, 
we obtain a Seifert surface of $K^{p, q}$ with genus $p \, g(F_K) + (p-1)(q-1)/2 = p \, g(K) + (p-1)(q-1)/2$ as desired. 

We may adapt this construction to obtain a bound on the slice genus of $K^{p,q}$.  
Let $S_K \subset B^4$ be a slice surface for $K$, and take a tubular neighborhood $S_K \times D^2$ in $B^4$.  
Since $S_{p,q} \subset S^3 = \bdry B^4$ and $\bdry S_{p,q}$ consists of the knot $K^{p,q}$ and $p$ parallel copies of the preferred longitude of $K$, 
we may now instead attach $p$ pairwise disjoint copies of $S_K$ in $S_K \times D^2 \subset B^4$ to $S_{p,q}$ to form a slice surface for $K^{p,q}$.  
Therefore we have 
\[g_4(K^{p,q})
 \leq p \, g(S_K) + (p-1)(q-1)/2
= p \, g_4(K) + (p-1)(q-1)/2.\]

\smallskip

Furthermore, if $c$ is an unknot disjoint from $K$, 
then observe that $\wind_c(K^{p,q}) = p\, \wind_c(K)$ and $\wrap_c(K^{p,q}) = p\, \wrap_c(K)$.  

Using how slopes on $K$ behave under twisting along $c$, 
we can determine how twist families of cables and cables of twist families are related.  
In particular, given a knot $K$ and twisting circle $c$, 
we write $(K^{p,q})_n$ for the $n$--twist of the $(p,q)$--cable of $K$  and $(K_n)^{p',q'}$ for the $(p',q')$--cable of the $n$--twist of $K$. 
If $ (K^{p,q})_n = (K_n)^{p',q'}$, 
then we must have $p=p'$ due to the wrapping numbers.  
Then since $r_n =r + \omega^2 n$ for slopes 
$r$ in $\bdry N(K)$ and $r_n$ in $\bdry N(K_n)$ using standard meridian-longitude coordinates, 
we have that $q' = q+p \omega^2 n$. 

Fix positive coprime integers $p,q$, 
and consider the family of knots $\{(K^{p,q})_n\}$ obtained by twisting the $(p,q)$--cable of $K$ about the unknot $c$.
Because $ (K^{p,q})_n = (K_n)^{p,q+p \omega^2 n}$, 
by the behavior of $\tau$  under cabling given in \cite{VanCott} and that $\tau$ gives a lower bound on $g_4$, 
we have 
\[p\, \tau( K_n ) + (p-1)(q+p \omega^2 n -1)/2 \leq \tau( (K^{p,q})_n ) \leq g_4( (K^{p,q})_n ) \leq p \, g_4(K_n) + (p-1)(q+p \omega^2 n -1)/2.\]
Therefore
\[ \frac{p \, g(K_n) + (p-1)(q+p \omega^2 n-1)/2}{p \, g_4(K_n) + (p-1)(q+p \omega^2 n -1)/2} \leq \frac{g( (K^{p,q})_n )}{g_4(  (K^{p,q})_n )} \leq \frac{p  \,g(K_n) + (p-1)(q+p \omega^2 n-1)/2}{p\, \tau( K_n ) + (p-1)(q+p \omega^2 n -1)/2}\]
and hence
\[ \frac{\frac{2p}{p-1} \frac{g(K_n)}{p \omega^2 n-1}+1}{\frac{2p}{p-1} \frac{g_4(K_n)}{p \omega^2 n-1}+1} \leq \frac{g( (K^{p,q})_n )}{g_4(  (K^{p,q})_n )} \leq \frac{\frac{2p}{p-1}  \frac{g(K_n)}{p \omega^2 n-1}+1}{\frac{2p}{p-1} \frac{\tau(K_n)}{p \omega^2 n-1}+1}.\]

Assume $g_4(K_n)$ is bounded as $n \to \infty$ so that $\tau(K_n)$ is bounded too.  
Then the denominators on the left and right above both decrease to $1$ as $n \to \infty$.  
Since the numerators are the same, we obtain

\[ \lim_{n \to \infty} \frac{g( (K^{p,q})_n )}{g_4(  (K^{p,q})_n )}  =1+ \frac{p}{p-1} \lim_{n \to \infty} \frac{2g(K_n)}{p \omega^2 n-1}.\]
\end{proof}

\begin{proof}[Proof of Theorem~\ref{thm:proportion}] 
We first construct a family of twist families of ribbon knots 
$\{K^m_n\} = \{ (K^m)_n \}$  
obtained by twisting $K^m$ along an unknot $c^m$ 
with $\omega = \wind_{c^m}(K^m) = 1$ so that $g_4(K^m_n) = 0$ and $g(K^m_n) = mn+1$ for positive integers $m,n>0$.  
Then we will apply Lemma~\ref{lem:boundedslicelimit} to obtain our result.

Consider the $3$--component link $K \cup \hat{c} \cup c$ in the top left of Figure~\ref{fig:wrappingribbons}, 
which may be recognized as the link $L10n73$ in the Thistlethwaite Link Table \cite{Thistlethwaite_katlas} and has Alexander polynomial 
\[\Delta_{K\, \cup\, \hat{c}\, \cup\, c}(x,y,z) = (-1 + x + y) (-x - y + x y) (-1 + z).\]
Let $E = S^3 - \nbhd(K \cup \hat{c} \cup c)$ be the exterior of this link.  
Then $H_1(E) = \Z^3$ generated by the homology classes of the oriented meridians $\mu_K$, $\mu_{\hat{c}}$, 
and $\mu_c$ of the components $K$, $\hat{c}$, and $c$, 
respectively.  
Let $\lambda_K$, $\lambda_{\hat{c}}$, and $\lambda_c$ be the corresponding oriented longitudes.
Observe that $[\lambda_{\hat{c}}] = -[\mu_c]$ and $[\lambda_c] = [\mu_K] - [\mu_{\hat{c}}]$ in $H_1(E)$.   
\\
\\
Now consider the family of links $K^m_n \cup \hat{c}^m_n \cup c^m_n$ 
(as shown on the right side of Figure~\ref{fig:wrappingribbons}) surgery dual to 
the filling of $E$ along the slopes 
$\mu_{K^m_n} = \mu_K$, $\mu_{\hat{c}^m}=\mu_{\hat{c}} - m \lambda_{\hat{c}}$, 
and $\mu_{c^m_n} = (mn+1)\mu_c - n\lambda_c$. 
That is, the new meridians are 
\begin{itemize}
	\item $[\mu_{K^m_n}] = [\mu_K]$,
	\item $[\mu_{\hat{c}^m_n}]=[\mu_{\hat{c}}] - m [\lambda_{\hat{c}}] = [\mu_{\hat{c}}] + m [\mu_c]$, and
	\item $[\mu_{c^m_n}] = (mn+1)[\mu_c] - n[\lambda_c] =  -n[\mu_K] + n[\mu_{\hat{c}}] + (mn+1)[\mu_c]$.
\end{itemize}
Hence 
\begin{itemize}
	\item $[\mu_K] = [\mu_{K^m_n}]$,
	\item $[\mu_{\hat{c}}] = -m n [\mu_{K^m_n}] + (mn+1)[\mu_{\hat{c}^m_n}] - m [\mu_{c^m_n}]$, and
	\item $[\mu_{c}] =  n [\mu_{K^m_n}] - n[\mu_{\hat{c}^m_n}] +  [\mu_{c^m_n}]$.
\end{itemize}
So we have: 
\begin{align*}
\Delta_{K^m_n\, \cup\, \hat{c}^m_n\, \cup\, c^m_n}(x,y,z)
& = \Delta_{K\, \cup\, \hat{c}\, \cup\, c}(x,  x^{-mn} y^{mn+1} z^{-m}, x^{n} y^{-n} z) \\
& = (-1 + x + x^{-mn} y^{mn+1} z^{-m}) (-x - x^{-mn} y^{mn+1} z^{-m} + x^{-mn+1} y^{mn+1} z^{-m}) (-1 + x^{n} y^{-n} z)
\end{align*}
Since 
$lk(K^m_n,\hat{c}^m_n)= -n$ and $lk(c^m_n,\hat{c}^m_n)=-1$, 
the 2nd Torres Condition \cite{Torres} implies

\begin{align*}
\Delta_{K^m_n\, \cup\, c^m_n}(x,z) 
&= \Delta_{K^m_n\, \cup\, \hat{c}^m_n\, \cup\, c^m_n}(x,1,z) / (x^{n} z^{-1}-1) \\
& = (-1 + x + x^{-mn}z^{-m})(-x-x^{-mn}z^{-m} + x^{-mn+1}z^{-m})(-1 + x^nz) / (x^{-n} z^{-1}-1) \\
& = (-1 + x + x^{-mn}z^{-m})(-x-x^{-mn}z^{-m} + x^{-mn+1}z^{-m})(x^{-n}z^{-1}-1)(-x^nz) / (x^{-n} z^{-1}-1)\\
& \doteq (-1 + x + x^{-mn}z^{-m})(-x-x^{-mn}z^{-m} + x^{-mn+1}z^{-m}).
\end{align*}
(Recall that $\doteq$ means equivalence up to a multiplicative unit.)
Then since $lk(K^m_n,c^m_n)=1$, the 2nd Torres Condition \cite{Torres} next implies

\begin{align*}
	\Delta_{K^m_n}(x) 
		&=\Delta_{K^m_n\, \cup\, c^m_n}(x,1) \cdot (x-1)/(x-1) \\
		&=\Delta_{K^m_n\, \cup\, c^m_n}(x,1) \\
		&= \left(x^{-m n}+x-1\right) \left(-x^{-m n}+x^{1-m n}-x\right)\\
                 &\doteq 1-x-x^{m n}+3 x^{m n+1}-x^{m n+2}-x^{2 m n+1}+x^{2 m n+2}\\
\end{align*}

Hence, 
$g(K^m_n) \ge mn + 1$. 
On the other hand, 
$K^m_n$ spans a ribbon disk with $mn+1$ ribbon singularities, 
i.e.\ ribbon number $mn+1$ as in the right of Figure~\ref{fig:wrappingribbons}.
Thus $mn + 1 \ge g(K^m_n)$ \cite{Fox}. 
It follows that $g(K^m_n) = mn + 1$. 

Then
by Lemma~\ref{lem:boundedslicelimit}, since $g_4(K^m_n) = 0$, for any choice of coprime positive integers $p,q$
 we have 
\[ \lim_{n \to \infty} \frac{g( ((K^m)^{p,q})_n )}{g_4(  ((K^m)^{p,q})_n )} 
=1+ \frac{p}{p-1} \lim_{n \to \infty} \frac{2g(K^m_n)}{p n-1}
= 1+ \frac{p}{p-1} \lim_{n \to \infty} \frac{2mn+2}{p n-1}
= 1+ \frac{2m}{p-1}\]
for each $m>0$.

For a given rational number $r > 1$, 
we can choose $m$ and $p$ so that $1+ \frac{2m}{p-1} = r$. 
This completes the proof of Theorem~\ref{thm:proportion}. 
\end{proof}

\begin{figure}[h]
	\begin{center}
		\includegraphics[width=0.75\textwidth]{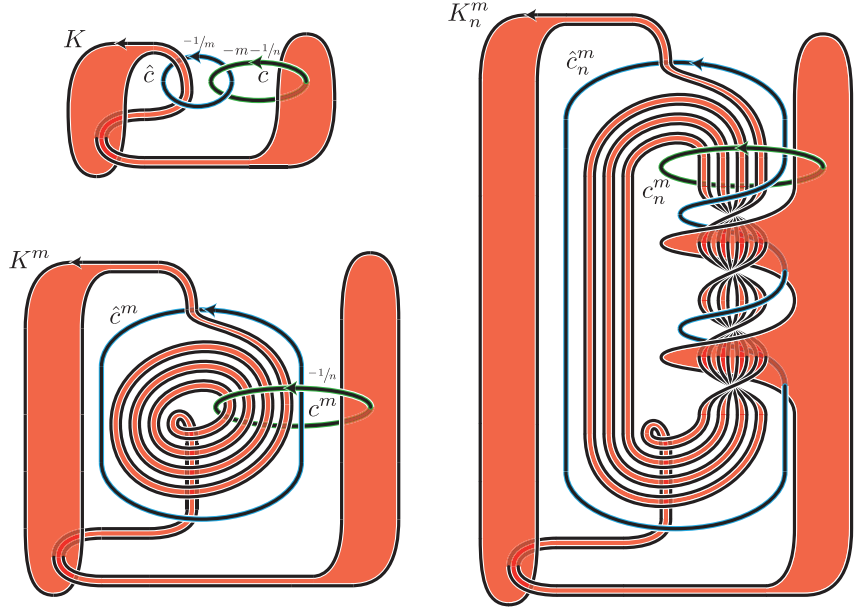}
		\caption{(Top Left) The link $K \cup \hat{c} \cup c$. 
		(Bottom Left) The result of $-1/m$--surgery on $\hat{c}$ yields the link $K^m \cup \hat{c}^m \cup c^m$.
		(Right) Twisting $K^m$ along $c^m$ produces the twist family $\{K^m_n\}$ with the link $K^m_n \cup \hat{c}^m_n \cup c^m_n$.
		The example shown has $m=4$ and $n=2$.}
		\label{fig:wrappingribbons}
	\end{center}
\end{figure}

\begin{proposition}
\label{prop:w=0}
For any $1 \le r \in \Q \cup \{\infty\}$, 
there exists a twist family $\{K_n\}$ with $\omega=0$ which satisfies 
$\displaystyle \lim_{n \to \infty} \frac{g(K_n)}{g_4(K_n)} = r$.
\end{proposition}

\begin{proof}	
We show that in fact for any $1 \le r \in \Q \cup \{\infty\}$ there are twist families with $\omega=0$ satisfying $\frac{g(K_n)}{g_4(K_n)} = r$ for all but finitely many integers $n$.
	
\smallskip
For $r = \infty$, let $K \cup c$ be a non-split link consisting of a non-trivial ribbon knot $K$ and an unknot $c$ that is disjoint from a ribbon disk for $K$.  
Then the twist family $\{K_n\}$ consists of ribbon knots so that $g_4(K_n) = 0$ for all $n \in \Z$ while $g(K_n)=0$ for at most one integer $n$ 
by \cite{KMS, Mathieu}.  
Hence for all but finitely many $n$, $g(K_n)/g_4(K_n)$ is a positive integer divided by $0$ and hence $g(K_n)/g_4(K_n)=\infty$ for these $n$ by convention.
	
\smallskip
For rational numbers $r \geq 1$, 
first let $T \cup c$ be the Whitehead link so that the twist family $\{T_n\}$ obtained by twisting $T$ along $c$ is the family of twist knots (where the Whitehead link is clasped positively so that $T_1$ is the figure eight knot, 
$T_{-1}$ is the positive trefoil knot, 
and $T_2$ is the knot $6_1$ in Rolfsen's knot table).  
Note that $g(T_n) = 1$ is except when $n = 0$ where $g(T_0)=0$, 
and Casson-Gordon \cite{CG} shows that $g_4(T_n)=1$ except when $n = 0, 2$ where $g_4(T_0)=g_4(T_2)=0$.

Then for any pair of non-negative integers $N,m$,  
let $K$ be the connected sum $T\, \#\, T_{2,2m+1}\, \#\, k$, 
where $T_{2,2m+1}$ is the $(2,2m+1)$--torus knot and $k$ is a slice knot with $g(k)=N$.  
(For example, one may take $k$ to be the connected sum of $N$ copies of $T_2$.)  
Then $K_n = T_n\, \#\, T_{2,2m+1}\, \#\, k$ so that $g(K_n) = 1 + m + N$ for $n \neq 0$.
Now we claim $g_4(K_n) = 1+ m$ for $n \le -1$. 
Since $k$ is slice, $g_4(K_n) \leq g(T_n\, \#\, T_{2,2m+1}) = 1+m$.  
Recall that the concordance invariant $\tau$ is additive under connected sums, so
$\tau(K_n) = \tau(T_n\, \#\, T_{2,2m+1}\, \#\, k) = \tau(T_n) + \tau(T_{2, 2m+1}) + \tau(k)$. 
Also note that $\tau(T_{2, 2m+1}) = m$ and $\tau(k) = 0$, because $k$ is slice. 
Furthermore following \cite[Theorem 1.5]{Hedden_Whitehead} we have $\tau(T_n) = 1$ when $n \le -1$ 
(which also follows from $T_n$ being strongly quasipositive for $n \leq -1$ \cite{Rudolph_braidzel}), 
and otherwise it is $0$. 
Thus $\tau(K_n) = 1+m \le g_4(K_n)$ for $n \le -1$. 
This shows that $g_4(K_n) = 1 + m$ whenever $n \le -1$.
Hence, for the mirrored twist family $\{\overline{K}_n\} = \{\overline{K_{-n}}\}$ we have  
$g(\overline{K_n}) = 1 + m + N$ and $g_4(\overline{K_n}) = 1 + m$, 
and 
$g(\overline{K}_n)/g_4(\overline{K}_n) = 1 +\frac{N}{1+m}$ for every $n \geq 1$.  
Since for any rational number $r \geq 1$ there are non-negative integers $N,m$ so that $r=1 +\frac{N}{1+m}$, 
we have our result. 
\end{proof}

\section{Tight fibered knots in twist families}
\label{tight_fibered_family}

\subsection{Positive twists, negative twists and coherent families}

Recall that  {\em tight fibered} knots are precisely the fibered, strongly quasipositive knots \cite{Hed_positive, BI}.

\begin{thm:postwistquasipositive}
Let $\{ K_n \}$ be a twist family of knots obtained by twisting $K$ along $c$ such that $c$ is neither split from $K$ nor a meridian of $K$. 
If $K_n$ is a tight fibered knot for infinitely many integers $n$, 
then $c$ links $K$ coherently and there is a constant $N$ such that 
\begin{itemize}
\item $K_n$ is tight fibered for every $n \geq N$ and
\item $K_n$ is tight fibered for only finitely many $n < N$.
\end{itemize}
\end{thm:postwistquasipositive}

\begin{proof}  
By Corollary~\ref{cor:sqptwist}(2) since $\{ K_n\}$ contains infinitely many tight fibered knots, 
$c$ links $K$ coherently.  
Then Theorem~\ref{thm:Murasugi-sum} shows the following:  
For any sufficiently large integer $N$ there is a minimal genus Seifert surface $\Sigma_N$ for $K_N$ 
such that for each $n > N$ the knot $K_n$ has a Seifert surface $\Sigma_n$ that may be obtained by a Murasugi sum of 
$\Sigma_N$ with the fiber surface of a positive torus knot.   
Through mirroring,
this theorem also implies that for any sufficiently negative integer $-N$ 
there is a minimal genus Seifert surface $\Sigma_{-N}$ for $K_{-N}$ such that the knot $K_n$ has a Seifert surface $\Sigma_n$ that may be obtained by a Murasugi sum of $\Sigma_{-N}$ with the fiber surface of a negative torus knot for each $n < -N$.  
Since a fibered knot has a unique minimal genus Seifert surface (up to isotopy), 
this sets us up to employ a key result of Rudolph that the Murasugi sum of two surfaces is quasipositive if and only if the two summands are quasipositive \cite{Rudolph_quasipositive_plumbing} and of Gabai that the Murasugi sum of two surfaces is a fiber if and only if the two summands are fibers \cite{gabai-fibered}.
	
So suppose there are infinitely many fibered strongly quasipositive knots in $\{K_n\}_{n<0}$. 
Then there is a negative integer $-N$ such that both $K_{-N}$ is a fibered strongly quasipositive knot and, 
by Theorem~\ref{thm:Murasugi-sum}, 
every knot $K_n$ with $n<-N$ has a minimal genus Seifert surface $\Sigma_n$ obtained as a Murasugi sum of the quasipositive fiber surface $\Sigma_{-N}$ of $K_{-N}$ and the fiber surface of a negative torus knot.
Since the fiber surface of a negative torus knot is not quasipositive, the surface $\Sigma_n$ for $n<-N$ cannot be quasipositive \cite{Rudolph_quasipositive_plumbing}.   
Because the Murasugi sum of two fiber surfaces is again a fiber \cite{gabai-fibered}, 
$\Sigma_{n}$ must be the unique minimal genus of $K_n$.  
Therefore $K_n$ cannot be the boundary of a quasipositive surface, 
and hence $K_n$ cannot be strongly quasipositive 
for $n < -N$, 
a contradiction.  
	
Now by hypothesis and because only finitely many knots in $\{K_n\}_{n<0}$ may be strongly quasipositive, 
we may choose $N$ sufficiently large so that both $K_N$ is a fibered strongly quasipositive knot with fiber surface $\Sigma_N$ and 
Theorem~\ref{thm:Murasugi-sum} applies. Since the fiber surface of a positive torus knot is a quasipositive surface, 
$\Sigma_n$ is a quasipositive fiber surface for all $n>N$ due to \cite{Rudolph_quasipositive_plumbing} and \cite{gabai-fibered}.  
Hence every knot in $\{K_n\}_{n \geq N}$ is fibered and strongly quasipositive.
\end{proof}

\subsection{Tightness in twist families and braid axes}

\begin{thm:braidaxis}
Let $\{K_n\}$ be a twist family obtained by twisting a knot $K$ along an unknot $c$ that is neither split from $K$ nor a meridian of $K$.
Then $c$ is a braid axis of $K$ if and only if both $K_n$ and $\overline{K_{-n}}$ and tight fibered for sufficiently large $n$.
\end{thm:braidaxis}

\begin{proof}
Half of this is well known.  
If $c$ is a braid axis, 
then $K_n$ is the closure of a positive braid for sufficiently positive integers $n > 0$ 
and the closure of a negative braid for sufficiently negative integers $n < 0$.  
A closed positive braid, being a sequence of plumbings of positive Hopf bands onto the disk, is tight fibered.  
The mirror of a closed negative braid is a closed positive braid, and hence tight fibered.

\medskip

So assume  $\{K_n\}$ is a twist family with winding number $\omega$ 
for which $K_n$ is positive tight fibered for sufficiently positive integers $n > 0$ 
and negative tight fibered for sufficiently negative integers $n < 0$.  
By Corollary~\ref{cor:sqptwist}, the twist family is coherent, 
and with the hypotheses of the theorem, $\omega = \eta \geq 2$.  
Therefore there is a balanced, oriented $\omega$--strand tangle $\kappa$ in $D^2 \times [0,1]$ so that the braid closure of $\kappa$ is $K$ and the axis of the braid closure is $c$.  
Furthermore, let $\delta = \sigma_{\omega-1} \dots \sigma_{2} \sigma_1 \in B_\omega$ be the dual Garside element in the $\omega$--strand braid group so that $\delta^\omega =\Delta^2$, the square of the Garside element $\Delta$, is a positive full twist.
Then $K_n$ is the braid closure of $ \Delta^{2n}\kappa$ for any $n \in \Z$.  
Let $\overline{\kappa}$ be the mirror of $\kappa$ in $D^2 \times I$ across a disk $D^2 \times \{1/2\}$.  
Observe that if $\kappa$ were a braid, then $\overline{\kappa} = \kappa^{-1}$.  
So, regarding the mirroring of $K_n$ as occurring across the sphere containing $c$ that $K$ intersects $2\omega$ times, 
we may view $\overline{K_n}$ as the braid closure of $\overline{\kappa}\Delta^{-2n}$.

Now by Theorems~\ref{thm:postwistquasipositive} and \ref{thm:Murasugi-sum}, 
there is a constant $N_+>0$ such that for all $n \geq N_+$, $K_n$ is tight fibered 
and the fiber surface $\hat{F}_n$ of $K_n$ is a Murasugi sum of the fiber $\hat{F}_{N_+}$ of $K_{N_+}$ with the fiber 
$\hat{F}_{(\omega,\ (n-N_+)\omega + 1)}$ of the $(\omega,\ (n-N_+)\omega + 1)$--torus knot as described in the proof of Theorem~\ref{thm:Murasugi-sum}.  

Similarly, there is a constant $N_- < 0$ such that for all $n \leq N_-$, 
$\overline{K_n}$ is tight fibered  and the fiber surface $\overline{\hat{F}_n}$ of $\overline{K_n}$ is a Murasugi sum of the fiber $\overline{\hat{F}_{N_-}}$ of $\overline{K_{N_-}}$ with the fiber 
$\hat{F}_{(\omega,\ -(n-N_-)\omega + 1)}$ of the $(\omega,\ -(n-N_-)\omega + 1)$--torus knot.  
(Since $n < N_-$, this is a positive torus knot.)  

Since $K_{N_+ +1}$ is tight fibered, 
we may take the Murasugi sum of $\hat{F}_{N_+ +1}$ with the fiber $\hat{F}_{(\omega,3)}$ of the $(\omega, 3)$--torus link in a similar manner to produce a surface $\hat{F}_{N_+ +1}'$ that is a fiber of the tight fibered link that is the braid closure of 
$\delta^2 \Delta^{2N_+ + 2} \kappa$.   
A subsequent Murasugi sum of $\hat{F}_{N_+ +1}'$ with $\overline{\hat{F}_{N_- -1}}$ may be then performed so that the resulting surface $\hat{F}_*$ is a fiber of the tight fibered knot $K^*$ that is the braid closure of the balanced tangle 
$\delta \overline{\kappa} \Delta^{-2N_- +2} \delta^{-1} \cdot \delta \cdot \Delta^{2N_+ + 2} \kappa 
= \delta \overline{\kappa} \Delta^{2(N_+ - N_-) +4}  \kappa$. 
This construction is illustrated in Figure~\ref{fig:sumposneg} in the case $\omega =3$.  
Since $\Delta^2$ is central in $B_{\omega}$, 
we have 
$\delta \overline{\kappa} \Delta^{2(N_+ - N_-) +4}  \kappa 
= \delta  \Delta^{2(N_+ - N_-) +4} \overline{\kappa} \kappa$.

\begin{figure}[h]
\begin{center}
\includegraphics[width=0.9\textwidth]{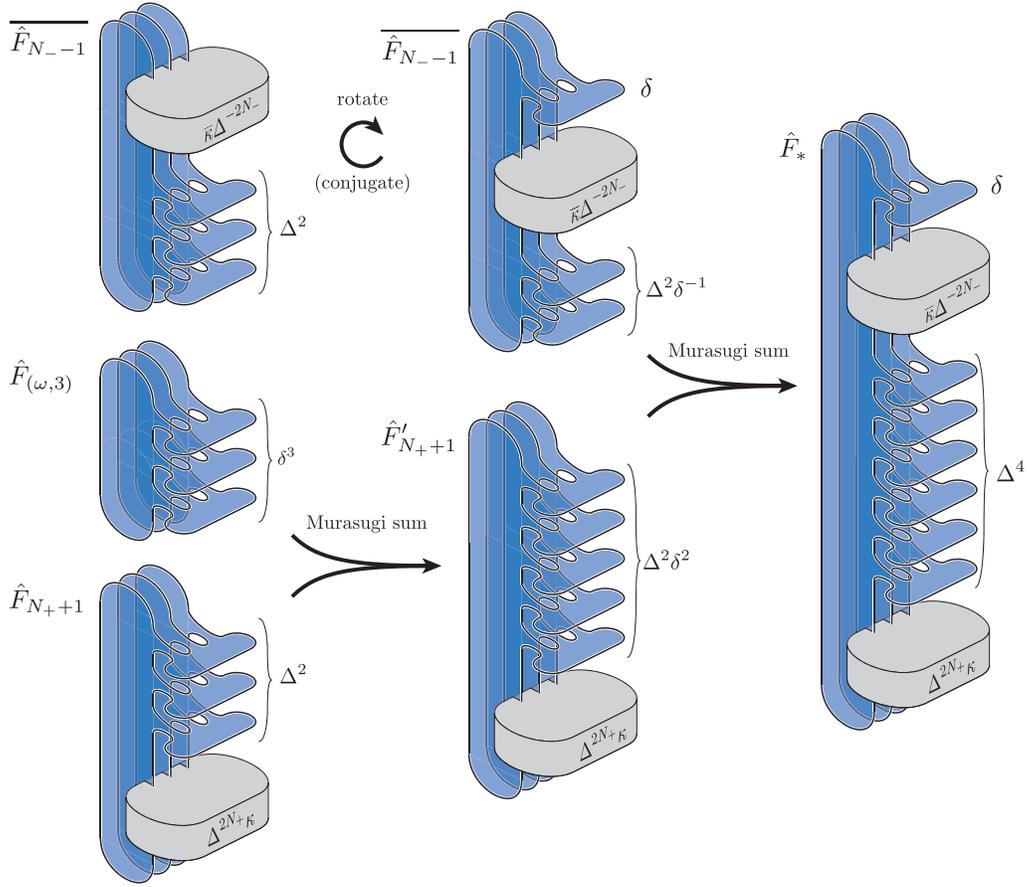}
\caption{Murasugi sums of the fibers $\hat{F}_{N_++1}$, $\overline{\hat{F}_{N_- - 1}}$, 
and the torus link fiber $\hat{F}_{(\omega,3)}$ produces the fiber $\hat{F}^*$. }
\label{fig:sumposneg}
\end{center}
\end{figure}

However, as in the proof of Proposition~\ref{prop:upperslicebound}, 
we may observe that $K^*$ is cobordant to the split link $D(\kappa) \sqcup T_{\omega,\ (N_+ - N_- + 2)\omega +1}$ 
by a planar surface with $\omega + 2$ boundary components.  
Here, $D(\kappa)$ is the double of $\kappa$, 
the braid closure of the balanced tangle $ \overline{\kappa}\kappa$, 
which is a link with $\omega$ components; 
see the top left of Figure~\ref{band_ribbon}. 

In the same way that one shows the connected sum of a knot and its mirror form a ribbon knot, 
we can see that $D(\kappa)$ is a ribbon link. 
Using the reflective symmetry of $D(\kappa)$, 
we may take a symmetric band connecting a local maximum and its reflected the local minimum 
to reduce the number of local maxima and local minima of $D(\kappa)$; 
by construction, this band connects the same component of $D(\kappa)$. 
Continuing such band surgeries, 
we obtain a link which is symmetric and 
every component becomes a trivial knot. 
Thus we obtain a trivial link so that $D(\kappa)$ bounds $\omega$ ribbon disks in $B^4$. 
See Figure~\ref{band_ribbon}.

\begin{figure}[h]
\begin{center}
\includegraphics[width=0.67\textwidth]{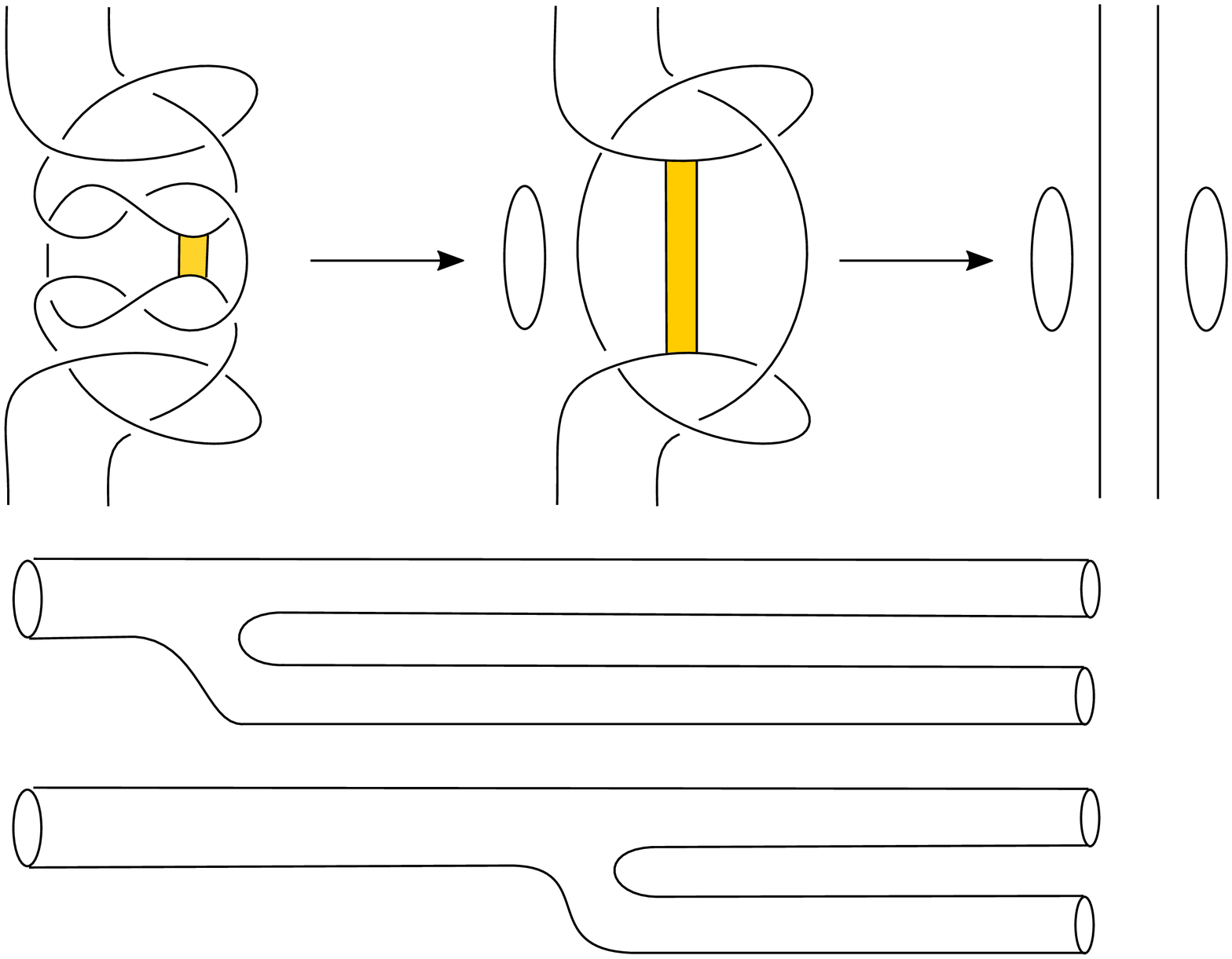}
\caption{$D(\kappa)$ is a ribbon link.}
\label{band_ribbon}
\end{center}
\end{figure}

Hence $K^*$ is actually {\em ribbon concordant} to the torus knot  $T_{\omega, (N_+ - N_- + 2)\omega +1}$. 
In particular, the knot 
$K^*\ \#\ \overline{ T_{\omega, (N_+ - N_- + 2)\omega +1}}$ is a ribbon knot.  
Therefore
\cite[Theorem 3]{Baker-NoteOnConcordance} tells us that $K^* = T_{\omega,\ (N_+ - N_- + 2)\omega +1}$.  

\medskip

Now let $c^*$ be the axis for the braid closure of $\delta  \Delta^{2(N_+ - N_- + 2)} \overline{\kappa} \kappa$, 
and consider the link $K^* \cup c^*$. 
Hence we obtain a twist family $\{K^*_n\}$ of braid closures of $\delta  \Delta^{2(N_+ - N_- + 2) } \overline{\kappa} \kappa$ 
where we know $K^*_n =  T_{\omega,\ n\omega +1}$ when $n \geq N_+ - N_- + 2$.  
(Here we parameterize the twist family so that $K^* = K^*_{N_+ - N_- + 2}$.)  
Lemma~\ref{lem:torusknottwistfamily} now implies that $K^*$ is the torus knot $T_{p,q}$ in the solid torus $S^3 - \nbhd(c^*)$ for some coprime integers $p$ and $q$, and hence $K^*_n = T_{p,\ np+q}$ for all integers $n$.  
Since  $K^*_n =  T_{\omega,\ n\omega +1}$ for $n \geq N_+ - N_- + 2$, 
it must be that in fact $K^*$ is isotopic to the torus knot $T_{\omega,\ (N_+ - N_- +2)\omega +1)}$ in the exterior of $c^*$.

In particular this means that the braid closure of the balanced $\omega$--tangle 
$\delta \Delta^{2(N_+ - N_-) +4} \overline{\kappa} \kappa$ is isotopic to the braid closure of 
$\delta \Delta^{2(N_+ - N_-) +4}$ in the exterior of $c^*$.   
Since $K^*$,  
the braid closure of $\delta  \Delta^{2(N_+ - N_-) +4} \overline{\kappa} \kappa$, 
is actually a closed $\omega$--braid, 
Lemma~\ref{lem:braidedtangleclosure} implies that $\delta \Delta^{2(N_+ - N_-) +4} \overline{\kappa} \kappa$ is also a braid. 
Then it follows from Lemma~\ref{lem:braidedtangleproduct} that $\kappa$ is a braid.  
Since $c$ is the axis of the braid closure of $\kappa$, $c$ is a braid axis of $K$.
\end{proof}

\begin{lemma}
\label{lem:torusknottwistfamily}
Let $K$ be a knot with a twisting circle $c$.  
If $K$ is not a $0$--bridge braid in the solid torus $S^3 - \nbhd(c)$, 
then there are at most five integers $n$ such that $K_n$ is a torus knot in $S^3$. 
\end{lemma}
\begin{proof}
For simplicity, we may assume $K=K_0$ is a torus knot by reparametrization. 
We divide the argument into three cases according to whether $S^3 - \nbhd(K \cup c)$ is hyperbolic, toroidal, or Seifert fibered. 
Assume that there are more than five integers $n$ such that $K_n$ is a torus knot. 

Suppose first that $S^3 - \nbhd(K \cup c)$ is hyperbolic. 
Then \cite[Proposition 5.11(2)]{DMM1} shows that $K_n$ is hyperbolic whenever $|n|> 3$, 
and hence there are at most $5$ integers $n$ (including $n = 0$) for which $K_n$ is a torus knot, 
contradicting the assumption.

Now let us assume that $S^3 - \nbhd(K \cup c)$ is toroidal.  
Consider the torus decomposition \cite{JS, Jo} of $X = S^3 - \nbhd(K \cup c)$, 
and let $M$ be the decomposing piece which contains $\partial N(c)$. 
If $M$ is not cabled, 
\cite[Theorem~2.0.1]{CGLS} shows that except for at most two integers $n$, 
the exterior $E(K_n) = X \cup _{-1/n} N(c)$ is toroidal, and $K_n$ is a satellite knot.  
If $M$ is a cable space, 
then we have a (smaller) solid torus $W$ in the solid torus $V = S^3 - \nbhd(c)$ complementary to $M$ which contains $K$ in its interior. 
Note that $W$ wraps at least twice in $V$ and the core of $W$ is a $0$--bridge braid in $V$. 
Hence, at most two integers $n$ can make the core of $W$ an unknot in $S^3$, 
i.e.\ $M \cup_{-1/n} N(c)$ is a solid torus. 
For other integers $n$, the core of $W$ is a nontrivial torus knot in $S^3$, 
i.e.\ $M \cup_{-1/n} N(c)$ is a torus knot space.
(For instance if the core of $W$ is the $(2, 1)$--cable in $V$, 
then $K = K_0$ has a companion $(2, 1)$--torus knot which is the unknot and $K_{-1}$ has a companion $(2, -1)$--torus knot which is again the unknot. 
Except this situation, there is at most one integer $n$ such that $M \cup_{-1/n} N(c)$ is a solid torus.)
Hence, for all but at most two integers $n$, 
$K_n$ is a satellite knot with a nontrivial torus knot as a companion. 
This contradicts the assumption. 

Hence, $S^3 - \nbhd(K \cup c)$ is Seifert fibered. 
Thus it is a cable space and $K$ is a $0$--bridge braid in $V = S^3 - \nbhd(c)$,  as desired.
\end{proof}

\begin{lemma}
\label{lem:braidedtangleproduct}
Let $\tau_1$ and $\tau_2$ be balanced $\omega$--tangles.  
If their product $\tau_1 \tau_2$ is an $\omega$--braid, 
then each $\tau_1$ and $\tau_2$ is an $\omega$--braid.
\end{lemma}

\begin{proof}
Let $D$ be the disk in $D^2 \times I$ that separates the balanced $\omega$--tangle $\tau_1 \tau_2$ into the two balanced 
$\omega$--tangles $\tau_1$ and $\tau_2$.  
Since the product $\tau_1 \tau_2$ is an $\omega$--braid, 
its exterior $X=(D^2 \times I) - \nbhd(\tau_1\tau_2)$ fibers over the interval $[0,1]$ 
with fibers $P_t = P \times \{t\}$ where $P$ is a planar surface with $\omega +1$ boundary components.  
The boundary of $P_t$ has one component as an essential curve in the annulus $\bdry D^2 \times I$ and each of the other components is a meridian of its own strand of $\tau_1\tau_2$.  
Since $\bdry D$ is an essential curve in $\bdry D^2 \times I$ and each strand of $\tau_1\tau_2$ intersects $D$ exactly once transversally, 
the punctured disk $D' = D \cap X$ is both homeomorphic to a fiber $P_t$ and homologous in $X$ to the fiber $P_t$ for some $t \in (0,1)$.  
Therefore $D'$ is isotopic in $X$ (keeping $\bdry D'$ in $\bdry P \times I$) to $P_t$.  
Hence the fibration of $X$ on each side of $D'$ pulls back to fibrations on the exteriors of $\tau_1$ and $\tau_2$. 
Thus both $\tau_1$ and $\tau_2$ are $\omega$--braids.
\end{proof}

\begin{lemma}
\label{lem:braidedtangleclosure}
Let $\tau$ be a balanced $\omega$--tangle.  
If its closure $\widehat{\tau}$ in the solid torus is a closed $\omega$--braid, 
then $\tau$ is an $\omega$--braid.
\end{lemma}

\begin{proof}
As a closed $\omega$--braid in the solid torus $D^2 \times S^1$, 
$\widehat{\tau}$ intersects each disk $D_t = D^2 \times \{t\}$, 
$t \in S^1$, in $\omega$ points.  
Then $D_0$ splits $\widehat{\tau}$ into an $\omega$--braid $\tau'$ in $D^2 \times I$.  
So if $D$ is the meridional disk in $D^2 \times S^1$ that splits open $\widehat{\tau}$ into the tangle $\tau$, 
it must have essential boundary and intersect $\widehat{\tau}$ in $\omega$ points.  
By standard innermost disk arguments with applications of Lemma~\ref{lem:braidedtangleproduct}, 
while maintaining its essential boundary and number of intersections with $\widehat{\tau}$, 
one may isotope $D$ to first be disjoint from $D_0$ and then to be $D_0$.  
Thus the $\omega$--tangle $\tau$ is homeomorphic to the $\omega$--braid $\tau'$. 
\end{proof}

\section{L-space knots in twist families}
\label{L-space_knot_family}

\subsection{Twist families of L-space knots and their limits} 
\label{limits}

Twisting a knot $K$ along an unknot $c$ also twists the slopes in $\bdry N(K)$. 
Using the standard parameterization of slopes on null-homologous knots as the extended rational numbers $\Q \cup \{\infty\}$,  
a slope $r$ for $K=K_0$ twists to a slope $r_n = r_0 + n \omega^2$ for $K_n$ where $\omega = \wind_c(K)$.
Choosing a slope $r$ for $K$ thus enhances the twist family of knots $\{K_n\}$ to a twist family of knot-slope pairs $\{(K_n,\, r_n)\}$ and hence a family of surgered $3$--manifolds $\{K_n(r_n)\}$.  
We call a knot-slope pair $(K,\, r)$ an {\em L-space surgery} if $K(r)$ is an L-space.

Observe that $K_n(r_n) = (K \cup c)(r,\, -1/n)$.  
So let us consider the manifold $Y_r = (K \cup c)(r,\, \emptyset) = K(r) -\nbhd(c)$, 
the exterior of $c$ in the $r$--surgery on $K$.  
Furthermore, let us retain the parameterization of slopes in $\bdry Y_r$ from its identification as $\bdry N(c)$.   
Then $K_n(r_n) = Y_r (-1/n)$ and $K_{\infty}(r_{\infty}) = Y_r(0)$.

\begin{lemma}
\label{lem:triad}
Assume $r>0$. If $K_{\infty}(r_{\infty})$ and $K_n(r_n)$ are L-spaces, 
then so is $K_{n+1}(r_{n+1})$. 
\end{lemma}

\begin{proof}
Let $(\mu_K,\, \lambda_K)$ be a preferred meridian-longitude pair of $K$, 
and $(\mu_c,\, \lambda_c)$ a preferred meridian-longitude pair of $c$. 
Then $\lambda_c = \omega \mu_K$, $\lambda_K = \omega \mu_c$, 
and it is easy to see that 
the manifold $Y_{p/q}(m/n)=(K\cup c)(p/q,\, m/n)$ obtained by $(p/q,\, m/n)$--surgery on $K \cup c$, 
where the pairs $p,q$ and $m,n$ are each coprime, 
has 	
\[H_1(Y_{p/q}(m/n);\ \mathbb{Z}) = 
\big\langle \mu_k,\ \mu_c\ |\ p\mu_K + \omega q \mu_c = 0,\ \omega n \mu_K + m \mu_c = 0 \big\rangle.\] 
Thus, taking positive coprime integers $p, q$ and $n\geq0$, 
we have $| H_1(K_{\infty}(r_{\infty});\ \mathbb{Z})| = \omega^2 q$,  
$|H_1(K_n(r_n);\ \mathbb{Z})| = p+n\omega^2q$, 
and 
$|H_1(K_{n+1}(r_{n+1});\ \mathbb{Z})| = p+(n+1)\omega^2q$. 
Therefore we have the following equality: 
\[|H_1(K_{n+1}(r_{n+1});\ \mathbb{Z})| =  |H_1(K_n(r_n);\ \mathbb{Z})| +  | H_1(K_{\infty}(r_{\infty});\ \mathbb{Z})|.\]
Since $K_n(r_n)$ and $K_{\infty}(r_{\infty})$ are L-spaces by assumption, 
\cite[Proposition~2.1]{OS_lens} shows that $K_{n+1}(r_{n+1})$ is also an L-space.
\end{proof}

\begin{lemma}
\label{lem:knots=>surgeries}
Assume that $K_n$ is an L-space knot for infinitely many integers $n$ and $c$ is neither split from $K$ nor a meridian of $K$. 
\begin{enumerate}
\item
There is a constant $N > 0$ such that the L-space knots $K_n$ with $n\geq N$ are positive L-space knots 
and those with $n \leq -N$ are negative L-space knots. 
\item
For any choice of slope $r \neq \infty$,  
there is a constant $N_r \geq 0$ such that for any L-space knot $K_n$ with $|n| \geq N_r$, 
$K_n(r_n) = K_n(r + n\omega^2)$ is an L-space. 
\end{enumerate}	
\end{lemma}

Note that Theorem~\ref{thm:positivetwistLspaceknots} improves Lemma~\ref{lem:knots=>surgeries}.
However Lemma~\ref{lem:knots=>surgeries} is needed in our development of Theorems~\ref{thm:limit} and \ref{thm:positivetwistLspaceknots}.

\begin{proof}

We may assume, by taking mirrors if necessary, 
that $K_n$ is an L-space knot for infinitely many integers $n \ge 0$. 

\begin{claim*}
\label{infinitely_positive}
For infinitely many integers $n \ge 0$ these L-space knots are positive L-space knots.
\end{claim*}

\begin{proof}
Assume to the contrary there are only finitely many $n \ge 0$ such that 
these L-space knots are positive L-space knots. 
Then we have a constant $N > 0$ such that $K_n$ is a negative L-space knot for infinitely many $n \ge N$. 
Then the mirror image $\overline{K_n} = \overline{K}_{-n}$ of $K_n$ is a positive L-space knot, 
and hence a tight fibered knot, 
for infinitely many $n \ge N$. 
This contradicts Theorem~\ref{thm:postwistquasipositive} (for the twist family $\{ \overline{K}_{-n} \}$). 
\end{proof}

\smallskip

Therefore $K_n$ is a tight fibered knot for infinitely many integers $n$. 
Then Theorem~\ref{thm:postwistquasipositive} shows that 
there is a constant $N^+ \ge 0$ such that 
 $K_n$ is a tight fibered knot for $n \ge N^+$. 
Hence, an L-space knot $K_n$ with $n \ge N^+$ is tight fibered, 
and $\tau(K_n) > 0$ or $K_n$ is the unknot. 
Thus $K_n$ is a positive L-space knot for $n \ge N^+$. 
(A nontrivial negative L-space knot has a negative $\tau$--invariant.)

Since positive L-space knots are tight fibered knots, 
Theorem~\ref{thm:postwistquasipositive} implies the twist family is coherent with winding number $\omega \geq 2$. 
Hence in Theorem~\ref{thm:twistgenera} we may take $x([D])=\omega-1 \geq 1$ so that for some constant $G$ we have 
$2g(K_n) -1 = 2G -1 + \omega(\omega - 1)n$ for $n > N'$ for some constant $N' \ge  N^+$. 
This means that for these $n > N'$ the points $(n,\, 2g(K_n) -1)$ lie on the line $y =  \omega(\omega - 1)x +  2G -1$.  
On the other hand, since $r_n = r + \omega^2 n$, 
the points $(n,\, r_n)$ lie on the line $y = \omega^2 x + r$. 
Since $\omega \ge 2$, 
the slope $\omega(\omega - 1)$ is smaller than the slope $\omega^2$, 
and hence there is a constant $N_r^+ \ge N'$ such that $r_n \ge 2g(K_n) -1$ for $n \ge N_r^+$; 
see Figure~\ref{genus_slope_surgery_slope}.  
It then follows from \cite{OS_rational} that $(K_n, r_n)$ is a positive L-space surgery 
for any L-space knot $K_n$ with $n \ge N_r^+$. 
	
If $\{K_n\}$ also contains infinitely many L-space knots $K_n$ for $n < 0$, 
a similar argument (applied to the mirrored twist family) produces a constant $N^-$ 
so that $K_n$ is a negative L-space knot for any L-space knot $K_n$ with $n \leq -N^-$ and 
a constant $N_r^- \geq 0$ so that $(K_n,\, r_n)$ is a negative L-space surgery for any L-space knot with $n \leq -N_r^-$.  
(Note that $r_n < -(2g(K_n)-1) \Leftrightarrow -(r_n) > 2g(\overline{K_n}) -1 \Leftrightarrow (-r)_{-n} > 2 g(\overline{K}_{-n})-1$.  
The knot $\overline{K_n}$  is the mirror of $K_n$, the result of $n$ twists of $K$ along $c$; it is equal to the knot $\overline{K}_{-n}$ which is the result of $-n$ twists of $\overline{K}$ along $\overline{c}$.) 
If $\{K_n\}$ only contains finitely many L-space knots $K_n$ for $n < 0$, 
choose $N^-$ so that $K_n$ is not an L-space knot if $n \leq -N^-$. 
Then the first assertion follows by choosing $N =   \max\{N^+,\, N^-\}$, 
and the second assertion follow by choosing $N_r = \max\{N_r^+,\, N_r^-\}$. 
\end{proof}
	
\begin{figure}[h]
\begin{center}
\includegraphics[width=0.6\textwidth]{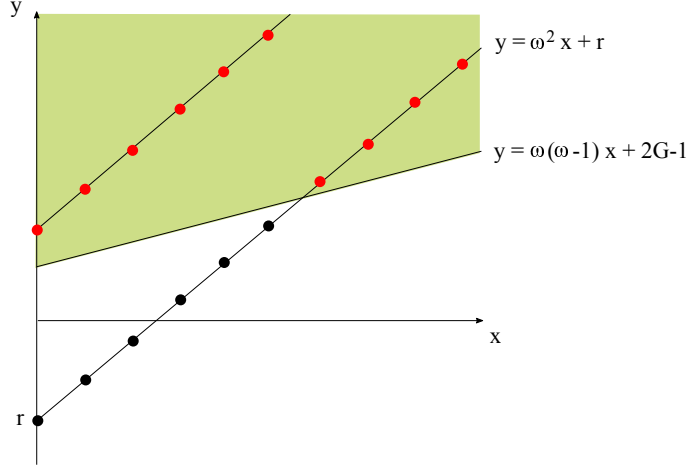}
\caption{The slope $\omega(\omega-1)$ is strictly smaller than $\omega^2$.}
\label{genus_slope_surgery_slope}
\end{center}
\end{figure}

Let $X$ be a connected, compact, oriented $3$--manifold where $\bdry X$ is a single torus.  
Let $\mathcal{L}(X)$ denote the subset of slopes $\alpha$ in $\bdry X$ such that $X(\alpha)$ is an L-space.  
A primitive element in $H_1(\partial X)$ is a {\em rational longitude} if it represents a torsion element when considered as an element of $H_1(X)$. 
Note that such an element is unique up to sign, 
and the resulting manifold obtained by Dehn filling along the rational longitude has the infinite first homology group. 

\begin{theorem}[{\cite[Proposition~1.3 and Theorem~1.6]{RR}}]
\label{thm:RRslopes}
The subset $\mathcal{L}(X)$ is either empty, a single slope, a closed interval of slopes, or the complement of the rational longitude.
\end{theorem}

\begin{thm:limit}
Let $\{ K_n \}$ be a twist family of knots obtained by twisting $K$ along $c$.
\begin{enumerate}
\item If $K=K_0$ is an L-space knot with L-space surgery slope $r$ and $K_{\infty}$ is an L-space knot, then 
\begin{enumerate}
\item $K_n$ is a positive L-space knot for all $n \geq 0$ if $r > 0$ and
\item $K_n$ is a negative L-space knot for all $n \leq 0$ if $r < 0$. 
\end{enumerate}
\item If $K_n$ is an L-space knot for infinitely many integers $n$, then $K_{\infty}$ is an L-space knot.
\end{enumerate}
\end{thm:limit}

\begin{proof}
(1) Assume that $K(r) = K_0(r_0) = Y_r(\infty)$ is an L-space, 
where $r \in \mathbb{Q}$, 
and $K_{\infty}$ is an L-space knot in $S^1 \times S^2$. 
Since $r_{\infty}$ is not a meridional slope on $K_{\infty}$, 
by Remark~\ref{L-space_knot_S1xS2}(1) $K_{\infty}(r_{\infty}) = Y_r(0)$ is also an L-space.
If $ r> 0$, repeated applications of Lemma~\ref{lem:triad} shows that $K_n(r_n)$ is an L-space for all $n \geq 0$.  
Because $r_n = r+n \omega^2$, $r_n >0$ and so $K_n$ is a positive L-space knot for all $n \geq 0$.   
Similarly, if $r<0$ then Lemma~\ref{lem:triad} implies that $K_n$ is a negative L-space knot for all $n\leq 0$.

(2) Assume $K_n$ is an L-space knot for infinitely many integers $n$. 
Then \cite[Theorem~1.5]{BM} implies $\omega >0$. 
Furthermore, Lemma~\ref{lem:knots=>surgeries} implies that for any slope $r \neq \infty$ (where $r=p/q$ for coprime $p,q$ and $q>0$) there is a constant $N$ so that for these $n$ with $|n|\geq N$, $K_n(r_n)=Y_r(-1/n)$ is an L-space.  
Hence for infinitely many integers $n$, the set $\{-1/n\}$ is contained in $\mathcal{L}(Y_r)$.  
Therefore, since $0$ is a limit point of this infinite set, 
Theorem~\ref{thm:RRslopes} implies that either $0 \in \mathcal{L}(Y_r)$ or $0$ is the rational longitude.  
However it cannot be the latter because $Y_r(0)=K_{\infty}(r_{\infty})$ implies that we have $|H_1(Y_r(0))|=\omega^2 q < \infty$ as calculated in the proof of Lemma~\ref{lem:triad}.  
Hence $K_{\infty}(r_{\infty})=Y_r(0)$ is an L-space and $K_{\infty}$ is an L-space knot.
\end{proof}

\smallskip

\begin{thm:positivetwistLspaceknots}
	Let $\{ K_n \}$ be a twist family of knots obtained by twisting $K$ along $c$ such that $c$ is neither split from $K$ nor a meridian of $K$.  If $K_n$ is an L-space knot for infinitely many integers $n>0$ \(resp.\ $n < 0$\), then 
	\begin{itemize}
		\item $c$ links $K$ coherently and
		\item there is a constant $N$ such that $K_n$ is a positive \(resp. negative\) L-space knot for all integers $n \geq N$ \(resp. $n \leq N$\).
	\end{itemize}
\end{thm:positivetwistLspaceknots}

\begin{proof}
By the assumption $\{ K_n \}$ or its mirror $\{ \overline{K}_n \}$ contains infinitely many positive L-space knots. 
Since positive L-space knots are tight fibered knots, 
Corollary~\ref{cor:sqptwist}(2) shows that $c$ links $K$ coherently. 

Assume $K_n$ is an L-space knot for infinitely many integers $n>0$. (The argument for $n<0$ follows similarly.)
Then Theorem~\ref{thm:limit}(2) shows that $K_{\infty}$ is an L-space knot (in $S^1 \times S^2$).
Also, Lemma~\ref{lem:knots=>surgeries} implies that all but finitely many of these knots are positive L-space knots.  
In particular, $K_{N}$ is a positive L-space knot for some $N>0$.  
Hence by Theorem~\ref{thm:limit}(1)(a) $K_n$ is a positive L-space knot for all $n \geq N$. 
\end{proof}

\begin{question}
\label{question_not_coherent}
Assume that $c$ does not link $K$ coherently. 
Then is there any universal upper bound (perhaps in terms of the winding and wrapping numbers) for the number of L-space knots in the twist family $\{ K_n \}$? 
\end{question}

\subsection{Two-sided infinite twist families of L-space knots}

Let $\{ K_n \}$ be a twist family of knots for which $K_n$ is an L-space knot for infinitely many integers $n$. 
Theorem~\ref{thm:positivetwistLspaceknots} asserts that $c$ must link $K$ coherently. 
We say that $\{ K_n \}$ is a \textit{two-sided infinite family \(of L-space knots\)} if $K_n$ are L-space knots for infinitely many positive integers $n$ and simultaneously $K_n$ are L-space knots for infinitely many negative integers $n$;
otherwise $\{ K_n \}$ is a \textit{one-sided infinite family}. 

\begin{proposition}
\label{positive_negative}
Let $\{ K_n \}$ be a twist family of knots which contains infinitely many L-space knots. 
Then the following three conditions are equivalent. 
\begin{enumerate}
\item
$\{ K_n \}$ is a two-sided infinite family. 

\item
$\{ K_n \}$ contains a positive L-space knot and a negative L-space knot. 

\item
There are constants $N_+$ and $N_-$ 
such that $K_n$ is a positive L-space knot for all $n \ge N_+$ and 
$K_n$ is a negative L-space knot for all $n \le N_-$. 
\end{enumerate}
\end{proposition}

\begin{proof}
$(1) \Rightarrow (3)$. 
Since $\{ K_n \}$ is a two-sided infinite family, 
it follows from Theorem~\ref{thm:positivetwistLspaceknots} that there is a constants $N_+$ and $N_-$ 
such that $K_n$ is a positive L-space knot for all $n \ge N_+$ and 
$K_n$ is a negative L-space knot for all $n \le N_-$, 
and the result follows.  

$(3) \Rightarrow (2)$. It is obvious. 

$(2) \Rightarrow (1)$. 
Since $\{K_n\}$ contains infinitely many L-space knots, 
Theorem~\ref{thm:limit}(2) shows that the limit $K_{\infty}$ is an L-space knot in $S^1 \times S^2$. 
Let $K_{p_0} \in \{ K_n \}$ be a positive L-space knot  and 
$K_{m_0} \in \{ K_n \}$ a negative L-space knot. 
Then $K_n$ is a positive L-space knot for all $n \ge p_0$, 
and $K_n$ is a negative L-space knot for all $n \le m_0$
by Theorem~\ref{thm:limit}(1)(a) and (b), respectively. 
Thus $\{ K_n \}$ is a two-sided infinite family. 
\end{proof}

As an immediate consequence of Theorem~\ref{thm:braidaxis}, we establish

\begin{corollary}
\label{cor:Lspacebraidaxis}
Let $\{K_n\}$ be a twist family that contains infinitely many L-space knots. 
If it is a two-sided infinite family, 
then the twisting circle $c$ is a braid axis for $K$ or $K \cup c$ is the unlink.
\end{corollary}

\begin{proof}
Since the mirrors of negative L-space knots are positive L-space knots, 
following Proposition~\ref{positive_negative}, 
$K_n$ and $\overline{K_{-n}}$ are a positive L-space knots for sufficiently large $n$. 

Since positive L-space knots are tight fibered, 
either $c$ is a braid axis of $K$, 
a meridian of $K$ or split from $K$ by Theorem~\ref{thm:braidaxis}.  
In the last two situations, $K_n = K$ for all $n$.  
Since the unknot is the only knot that is both a positive and a negative L-space knot, 
$K$ must be the unknot. 
If $c$ is a meridian of $K$, $c$ is a braid axis, 
and if $c$ is split from $K$, $K \cup c$ is the unlink.  
\end{proof}

It seems plausible that the converse of Corollary~\ref{cor:Lspacebraidaxis} holds. 
\begin{question}
\label{braid=>two-sided}
Let $\{ K_n \}$ be a twist family of knots with infinitely many L-space knots. 
If $c$ is a braid axis, then is $\{ K_n \}$ a two-sided infinite family?
\end{question}

\begin{example}
\label{P-231}
Let $K$ be the pretzel knot $P(-2, 3, 1) = T_{5, 2}$.  
Then $K_n$ is an L-space knot for $n \ge 0$ \cite{OS_lens}. 
On the other hand, 
following Lidman-Moore \cite[Theorem~2.1]{LM} $K_n$ is not an L-space knot if $n < 0$. 
In this example, 
$c$ is not a braid axis and $\{ K_n \}$ is a one-sided infinite twist family.  
\end{example}

\begin{figure}[h]
\begin{center}
\includegraphics[width=1.0\textwidth]{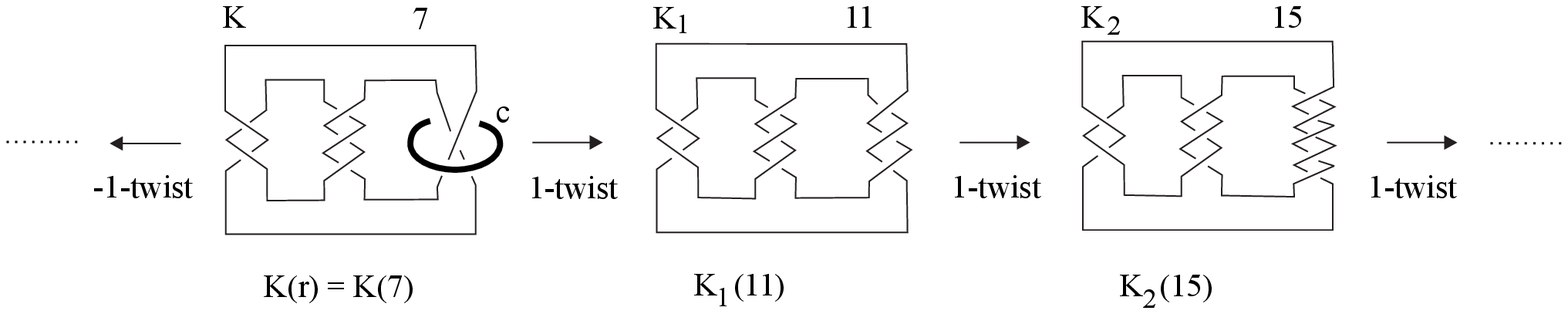}
\caption{$K_n$ is an L-space knot if and only if $n \ge 0$. }
\label{twist_non_braid}
\end{center}
\end{figure}

\begin{example}
All torus knots have twisting circles which produce two-sided infinite families of L-space knots 
and ``almost half'' of them have twisting circles which produce one-sided infinite families of L-space knots as well like the one in Example~\ref{P-231}. 
Take a torus knot $T_{p, q}$ which lies in a standardly embedded torus in $S^3$. 
Let $c_+$ and $c_-$ be unknots depicted in Figure~\ref{fig:torusknotseiferters}. 
Note that $T_{p, q} \cup c_+$ and $T_{-p, q} \cup c_-$ are mirror images of each other, 
$lk(T_{p, q}, c_+) = p + q$ and $lk(T_{p, q}, c_-) = p-q$.

\begin{figure}
\begin{center}
\includegraphics[width=\textwidth]{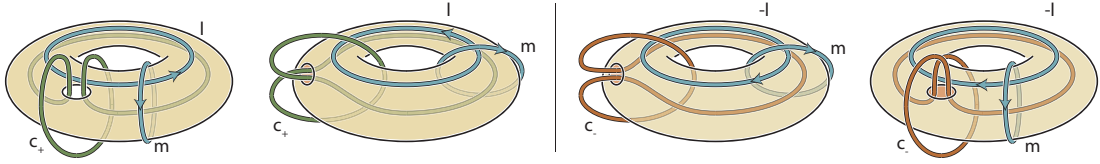}
\caption{The twisting circle $c_+$ for a $(p,q)$--torus knot is mirror equivalent to the twisting circle $c_-$ for a $(-p,q)$--torus knot.}
\label{fig:torusknotseiferters}
\end{center}
\end{figure}

We assume $2 \le q < p$ for simplicity. 
As shown in \cite[Proposition~6.3]{BM}, $c_+$ is a braid axis for $T_{p, q}$ and the knot $T^+_{p, q, n}$ obtained from $T_{p, q}$ 
by $n$--twist along $c_+$ is an L-space knot for all integers $n$. 
However, assume that $2 \le q < p < 2q$, and take $c_-$ instead of $c_+$. 
Then by \cite[Proposition~6.3]{BM} $c_-$ is not a braid axis of $T_{p, q}$, 
and the knot $T^-_{p, q, n}$ obtained from $T_{p, q}$ 
by $n$--twist along $c_-$ is an L-space knot for $n \ge -1$. 
Now Corollary~\ref{cor:Lspacebraidaxis} allows us to conclude that 
$T^-_{p,q,n}$ is an L-space knot for only finitely many $n \leq -2$.  
Hence $\{T^-_{p,q,n}\}$ is one-sided twist family. 
(Note that $K_2 \cup c$ in Example~\ref{P-231} is isotopic to $T_{5, 3} \cup c_-$.) 
\end{example} 

\begin{question}
\label{question_braid}
If $\{K_n\}$ is a twist family that contains a nontrivial positive L-space knot and a nontrivial negative L-space knot, 
then does it contain infinitely many L-space knots? 
\end{question}

\section{Satellite L-space knots}
\label{satellite}

Let $P(K)$ denote the \textit{satellite knot} with a \textit{companion knot} $K$ and \textit{pattern} $(V,\, P)$, 
where $P$ is a knot in the unknotted solid torus $V = S^3-\nbhd(c)$ for some unknot $c$. When we think of the pattern $P$ as a knot in $S^3$, we call $P$ a \textit{pattern knot}.  
In particular, 
$P(K)$ is the image of $P$ under the identification of $V$ containing $K$ with $N(K)$ so that the meridian of $c$ is identified with the preferred longitude of $K$.  
Further observe that the unknot $c$ linking with $P$ defines the twist family $\{P_n\}$ of pattern knots, 
inducing $n$--twisted satellites $P_n(K)$. 
In the following, we use the term satellite knot $P(K)$ to mean that the companion knot $K$ is a nontrivial knot and 
$P$ is neither embedded in a $3$--ball in $V$ nor a core of $V$, i.e.\ $P \cup c$ is neither the trivial link nor the Hopf link.

We say the satellite knot $P(K)$ is a {\em braided satellite} if $P$ is a closed braid in $V$, i.e.\ $c$ is a braid axis for $P$.  
In what follows we will observe that, assuming Conjecture~\ref{L-slopes},  
Theorem~\ref{braided pattern} shows that an L-space satellite knot of a non-trivial knot must be a braided satellite.  
Thus, under such an assumption, we affirmatively answer the first part of \cite[Question 22]{BakerMoore}.  
More broadly, without the need to assume the conjecture, 
together Proposition~\ref{positive_negative} and Corollary~\ref{cor:Lspacebraidaxis} affirmatively answer the first part of \cite[Question 1.8]{Hom-satellite}. 

\medskip

Let $X_K = S^3 - \nbhd(K)$ be the exterior of $K$, and let $V_{P} = S^3 - \nbhd(P \cup c)$ be the exterior of the link $P \cup c$.  
For a slope $\alpha$ in $\bdry N(P)$, 
let $V_P(\alpha)$ denote the Dehn filling of $V_P$ along that slope.
Then $P(K)(\alpha)$ is obtained by gluing $X_K$ and $V_P(\alpha)$ along their boundaries. 

Recall that for a $3$--manifold $M$ with torus boundary, 
$\mathcal{L}(M)$ is the set of slopes in $\bdry M$ along which Dehn filling yields an L-space. 
The interior of $\mathcal{L}(M)$ is denoted by $\mathcal{L}^{\circ}(M)$.  
Let $M_i$ \($i = 1, 2$\) be a $3$--manifold with a single torus boundary, 
and glue them along their boundaries via a homeomorphism $h : \partial M_1 \to \partial M_2$ 
to obtain a closed $3$--manifold $M_1 \cup_h M_2$.  
Rasmussen and Rasmussen \cite{RR} conjecture the following gluing condition 
for the resulting $3$--manifold $M_1 \cup_h M_2$ to be an L-space.

\begin{conjecture}[{\cite[Conjecture 1.7]{RR}}]
\label{L-slopes}
Assume $M_1$ and $M_2$ have incompressible boundary. 
The glued $3$--manifold $M_1 \cup_h M_2$ is an L-space if and only if 
$ h(\mathcal{L}^{\circ}(M_1)) \cup \mathcal{L}^{\circ}(M_2) = \mathbb{Q} \cup \{ 1/0 \}$. 
\end{conjecture}

\begin{remark}
For ``loop-type'' manifolds, 
Hanselman, Rasmussen and Watson \cite[Theorem~5]{HRW} establish Conjecture~\ref{L-slopes}. 
Watson \cite{Watson-NewtonInst} has further announced 
his recent joint work with Hanselman and Rasmussen which settles Conjecture~\ref{L-slopes} without the extra hypothesis of ``loop-type''. 
\end{remark}

Assuming Conjecture~\ref{L-slopes}, Hom \cite[Proposition~3.3]{Hom-satellite} proves the following result. 
Although she states the result in the case where $P(K)$ is a positive L-space knot, 
the case when $P(K)$ is a negative L-space knot follows immediately  since $\overline{P(K)} = \overline{P}(\overline{K})$.  

\begin{theorem}[Cf.\ {\cite[Proposition~3.3]{Hom-satellite}} and {\cite[Theorem 35]{HRW}}]
\label{Hom}
Suppose that Conjecture~\ref{L-slopes} is true.  
If a satellite knot $P(K)$ is a positive \(resp. negative\) L-space knot, 
then we have the following.
\begin{enumerate}
\item
$P$ and $K$ are positive \(resp. negative\) L-space knots,
\item
$P_n$ is a positive L-space knot for all $n \ge -2g(K) + 1$, 
and $P_{-n}$ is a negative L-space knot for all sufficiently large $n$
\(resp. $P_{-n}$ is a negative L-space knot for all $-n \le 2g(K) -1$, 
and $P_n$ is a positive L-space knot for all sufficiently large $n$\). 
\end{enumerate}
\end{theorem}

It follows from (1) that if $P(K)$ is a satellite L-space knot, 
then its pattern knot $P$ and companion knot $K$ must be L-space knots. 
However, it does not provide an information about the pattern $(V, P)$. 
We are able to put a further restriction on the pattern $(V,P)$ of a satellite L-space knot, 
which are conjectured in \cite{BakerMoore}(cf. \cite{Hom-satellite}). 

\begin{theorem}
\label{braided pattern}
Suppose that Conjecture~\ref{L-slopes} is true. 
If $P(K)$ is a satellite L-space knot,  
then it is a braided satellite knot, 
i.e.\ $P$ is a closed braid in $V$. 
\end{theorem}

\begin{proof}
Employing Theorem~\ref{Hom}(2), 
Corollary~\ref{cor:Lspacebraidaxis} implies that $c$ is a braid axis for $P$.  
Hence $P(K)$ is a braided satellite.
\end{proof}

Suppose that a satellite knot $P(K)$ is an L-space knot. 
Then Theorem~\ref{braided pattern} shows that $P$ is braided in $V$, 
and hence the winding number $\omega$ coincides with the wrapping number, 
which is also the braid index. 
In particular, $\omega \ge 2$ for a satellite L-space knot $P(K)$.   

Let $P(K)$ be an $(\omega, q)$--cable of a knot $K$, where $\omega \ge 2$. 
Then it follows from \cite[Theorem 1.10]{Hedden-cable} and \cite{Hom-cable} that 
$P(K)$ is a positive L-space knot if and only if $K$ is a positive L-space knot and $0< \omega (2g(K)-1) \le q$. 
Note that we must actually have the strict inequality $\omega (2g(K)-1) < q$.  
Indeed, if $\omega (2g(K)-1) = q$, then $q$ is divisible by $\omega$; 
yet since $\omega$ and $q$ are coprime we must have $\omega=1$ which is contrary to our assumption that $\omega \geq 2$. 
Since $2g(P) = (\omega -1 )(q-1)$, 
we may rewrite this as $q = \frac{2g(P)-1 +\omega}{\omega-1}$. 
Hence we obtain an inequality 
$\omega(2g(K)-1) < \frac{2g(P)-1 +\omega}{\omega-1}$. 

The next result shows this is always the case for satellite L-space knots $P(K)$, 
and gives a constraint between the genera of the pattern knot $P$ and the companion knot $K$. 

\begin{theorem}
\label{g(K)_g(P)_omega}
Suppose that Conjecture~\ref{L-slopes} is true.  
Assume a satellite knot $P(K)$ is an L-space knot. 
Then we have: 
\begin{enumerate}
\item
$\displaystyle\omega(2g(K)-1) < \frac{2g(P)-1 +\omega}{\omega-1}$, and 
\item
$g(K) < g(P)$, 
or $g(K)=g(P)=1$. 
In the latter case, 
$P(K)$ is the $(2,3)$--cable of the trefoil knot $T_{3, 2}$ 
\(or $(2,-3)$--cable of the trefoil knot $T_{-3, 2}$\). 
\end{enumerate}
\end{theorem}

\begin{proof}
We will prove the theorem in the case where $P(K)$ is a positive L-space knot. 
If $P(K)$ is a negative L-space knot, 
take the mirror $\overline{P(K)} = \overline{P}(\overline{K})$ of $P(K)$ and apply the same argument.
We begin with a proof of (1).  
If $P$ is a $0$--bridge braid in $V$, 
i.e.\ $P(K)$ is a cable knot of $K$, 
then as we mentioned above the first assertion follows from \cite{Hom-cable}. 
So in the following we assume $P$ is not a $0$--bridge braid in $V$.  

Given that $P(K)$ is a positive L-space knot, 
$P(K)(r)$, the glued $3$--manifold $X_K \cup_h V_P(r)$, 
is an L-space for any $r \geq 2g(P(K))-1$.  
Since $K$ is nontrivial, $X_K$ is boundary-irreducible, 
i.e.\ it has an incompressible boundary. 

\begin{claim*}
\label{boundary-irreducible}
$V_P(r)$ is boundary-irreducible when $r = 2g(P(K))-1$. 
\end{claim*}
 
\begin{proof}
Assume for a contradiction that $V_P(2g(P(K))-1)$ is boundary-reducible. 
Then it is either $S^1 \times D^2$ or $S^1 \times D^2\ \sharp\ W$ where $W$ is a nontrivial lens space \cite{Sch}. 

\smallskip

\noindent
\textbf{Case 1.}\ $V_P(2g(P(K))-1) = S^1 \times D^2$. 
Then $P$ is a $0$ or $1$--bridge braid in $V$ \cite{Gabai_solid_tori} . 
By our assumption that $P$ is not a $0$--bridge braid,  
we may assume it is a $1$--bridge braid in $V$. 
Let us use the updated notation of \cite{HLV} instead of what \cite{Gabai_braid} does.  
In particular, 
following \cite[p.2, Lemma~2.1 and footnote 3]{HLV}, 
using the bridge width $b$ and the twist number $t$, 
we have an expression $2g(P(K))-1 = t \omega + d$, 
where $d = b$ or $b+1$ and $1 \leq b \leq \omega -2,\  t = t_0 + q\omega$ for integers $q$ 
and constraining $t_0$ with $1 \leq t_0 \leq \omega -2$ to avoid $0$--bridge braids.   
Using these parameters, 
$2g(P) = (t-1)(\omega -1)+b$ \cite[Lemma~2.6]{HLV} and so $2g(P)-1 = t \omega - t - \omega + b$. 
Then we have 
\begin{align*}
t \omega + d
&= 2g(P(K))-1 \\
& = 2g(P) + 2\omega g(K) - 1 \\
&=  t \omega - t - \omega + b + \omega 2g(K). 
\end{align*}

Thus we have 
$t + (d-b) = \omega 2g(K) - \omega$. 
Then use that $t = t_0 + q \omega$ to obtain 
\[t_0 + (d-b) = \omega 2g(K) - \omega - q \omega 
= \omega(2g(K) -1 - q).\] 
However, 
$1 \leq t_0 \leq t_0 + (d-b) \leq t_0+1 \leq \omega -1$
so $t_0 + (d-b)$ cannot have $\omega$ as a factor, 
a contradiction.

\smallskip

\noindent
\textbf{Case 2.}\ $V_P(2g(P(K))-1) = S^1 \times D^2\ \sharp\ W$ where $W$ is a nontrivial lens space. 
Then following \cite{Sch}, 
$P$ is a $(p, q)$--cable of a knot $P'$ in $V$ and the surgery slope $r = 2g(P(K)) -1$ is the cabling slope $pq$, where $p \ge 2$. 
By the assumption $P'$ is not a core of $V$. 
Note that $V_P(pq) = V_{P'}(q/p)\ \sharp\ W$ for some lens space $W$ \cite{Go_satellite} 
and, since $V_{P'}(q/p)$ is irreducible but boundary-reducible, $V_{P'}(q/p) =S^1 \times D^2$. 
Since any non-integral surgery on a $1$--bridge braid never yields $S^1 \times D^2$ \cite[Lemma~3.2]{Gabai_braid}, 
$P'$ is a $0$--bridge braid in $V$, say $(r, s)$--torus knot in $V$, where $r \ge 2$. 
(Since $V_{P'}(q/p) = S^1 \times D^2$, we have $| rsp - q| = 1$, though this will not be used in the following.)
Hence $P$ is a $(p, q)$--cable of $T_{r, s}$. 
Then using \cite[\S21 Satz 1]{Schubert} 

\begin{align*}
pq 
& = 2g(P(K)) - 1 \\
& = 2g(T_{p, q}) + 2pg(T_{r, s}) -1 \\
& = p(q + rs -r -s) - q.
\end{align*}
The left-hand side is divided by $p \ge 2$, 
but the right-hand side cannot be divided by $p$, because $p$ and $q$ are coprime. 
\end{proof}

Thus both $X_K$ and $V_P(2g(P(K))-1)$ are boundary-irreducible. 
Since $\mathcal{L}(X_K) = [2g(K)-1, \infty]$ with respect to the standard $\mu_K, \lambda_K$ basis of $\bdry X_K$, 
we have that $h(\mathcal{L}^{\circ}(X_K)) = (0, \frac{1}{2g(K)-1})$ with respect to the standard $\mu_c, \lambda_c$ basis of 
$\bdry N(c) = \bdry V_P(r)$. 
Assuming Conjecture~\ref{L-slopes}, 
since $P(K)(2g(P(K))-1)$ is a positive L-space, 
we have 
\[ h(\mathcal{L}^{\circ}(X_K)) \cup \mathcal{L}^{\circ}(V_P(2g(P(K))-1)) 
=  (0, \frac{1}{2g(K)-1}) \cup \mathcal{L}^{\circ}(V_P(2g(P(K))-1)) 
= \mathbb{Q} \cup \{\infty\}. \]
Since the rational longitude of $V_P(2g(P(K))-1)$ has slope $\frac{\omega^2}{2g(P(K))-1}$ \cite[Lemma~3.3]{Go_satellite} 
and cannot be in $\mathcal{L}(V_P(2g(P(K))-1))$,  
it must be in $(0, \frac{1}{2g(K) -1})$,  
hence we have 
\[\frac{\omega^2}{2g(P(K)) -1} <\frac{1}{2g(K)-1}.\] 
It follows that 
\[\omega^2 (2g(K)-1) <  2g(P(K))-1.\]  
Since $P(K)$ is fibered, following \cite[Lemma~2.6]{HLV} we have: 
\[2g(P(K))-1 = 2g(P)-1 + \omega 2g(K),\] 
this inequality reduces to 
\[\omega(2g(K)-1) < \frac{2g(P)-1 +\omega}{\omega-1}\] 
as claimed.

\smallskip

Let us prove (2). 
If $\omega = 2$, 
then by (1) we have $2(2g(K)-1) < 2g(P)+1$ and hence $2g(K)-1 < g(P)+1/2$.  
Assuming that $g(P) \leq g(K)$, 
this then implies $2g(K)-1 < g(K)+1/2$ and so $g(K) < 3/2$.  
Since $K$ is non-trivial, $g(K)=1$ and $g(P) =1$ as well. 
Theorem~\ref{Hom} shows $P$ and $K$ are positive L-space knots, 
hence $K$ is the positive trefoil knot $T_{3, 2}$ \cite{Ghi}.  
By Theorem~\ref{braided pattern} $P$ is a closed braid in $P$ with braid index two, 
thus $P$ is a $(2, q)$--torus knot in $V$ for some $q$.
Since $g(P) = 1$, $q = 3$ and $P$ is a $(2, 3)$--torus knot in $V$. 
Hence, $P(K)$ is the $(2,3)$--cable of the trefoil knot $T_{3, 2}$.

So we may assume $\omega \ge 3$. 
If $g(P) = 0$, 
then (1) implies $\omega (2g(K) -1) < 1$, 
which means that $g(K) = 0$, 
i.e.\ $K$ is unknotted, a contradiction. 
So we assume $g(P) \ge 1$ and $\omega \ge 3$. 
Rewrite the inequality in (1) as 
\[2g(K) < \frac{2g(P) + \omega^2 - 1}{\omega(\omega-1)}\]
Since $\omega \ge 3$, 
we have $\omega^2 -1 < 2(\omega^2 - \omega -1)$. 
Hence $\omega^2 - 1 < 2(\omega ^2 - \omega -1)g(P)$.
Using this we have 
\[\frac{2g(P) + \omega^2 - 1}{\omega(\omega-1)} < 2g(P).\] 
Connecting the above two inequalities, we obtain the desired inequality $g(K) < g(P)$. 
\end{proof}

The second assertion in Theorem~\ref{g(K)_g(P)_omega} says that even when both $P$ and $K$ are L-space knots, 
we may not construct a satellite L-space knot for any (braided) embedding of $P$ in $V$.

\medskip

Finally we consider an essential tangle decomposition of satellite L-space knots. 
A knot $K$ in $S^3$ admits an \textit{essential $n$--string tangle decomposition} 
if there exists a $2$--sphere $S$ which intersects $K$ transversely in $2n$--points such that 
$S- \nbhd(K)$ is essential in $E(K)$. 
In particular, when $n = 2$, such a $2$--sphere is called an \textit{essential Conway sphere}. 
We say that $K$ has \textit{no essential tangle decomposition} 
if it has no essential $n$--string tangle decomposition for all integers $n$. 
Krcatovich \cite{Kr} shows that an L-space knot has no essential $1$--string tangle decomposition, 
i.e.\ $K$ is prime. 
Lidman and Moore conjecture \cite{LM} that any L-space knot $K$ admits no essential $2$--string tangle decomposition,
i.e.\ $K$ admits no essential Conway sphere. 
More generally, it is conjectured in \cite{BakerMoore}: 

\begin{conjecture}
\label{n-string prime}
Any L-space knot has no essential tangle decomposition.  
\end{conjecture}

We apply Theorem~\ref{braided pattern} to give a partial answer to Conjecture~\ref{n-string prime}. 
The first assertion in Theorem~\ref{satellite_n-string prime} immediately follows from Theorem~\ref{braided pattern} and 
\cite[Theorem~20]{BakerMoore}, 
but we will give its proof below in the course of the proof of (2). 

\begin{theorem}
\label{satellite_n-string prime}
Suppose that Conjecture~\ref{L-slopes} is true. 
Let $P(K)$ be a satellite L-space knot with a pattern $(V, P)$ and a companion knot $K$.
\begin{enumerate}
\item
If $K$ has no essential tangle decomposition, 
then $P(K)$ has no essential tangle decomposition neither. 
\item
$P(K)$ has no essential $n$--string tangle decomposition for $n \le 3$. 
In particular, $P(K)$ does not admit an essential Conway sphere. 
\end{enumerate}
\end{theorem}

\begin{proof}
Suppose that we have a $2$--sphere $S$ which gives an essential tangle decomposition for $P(K)$. 
It follows from Theorem~\ref{braided pattern} that 
$P$ is a closed braid in $V$ (with braid index $\omega \ge 2$), 
and hence $(V, P)$ has no essential tangle decomposition.  
Then following \cite{HMO}, 
$S$ gives an essential tangle decomposition of the companion knot $K$ for some integer $n$. 
This contradicts the assumption of (1), and completes a proof of (1). 

To proceed the proof of (2), 
assume for a contradiction that $S$ gives an essential $n$--string tangle decomposition for $P(K)$ for $n \le 3$. 
Then $S$ intersects $P(K)$ transversely in $2n \le 6$ points. 
As above $S$ gives an essential $n'$--string tangle decomposition of $K$ for some integer $n'$. 
By Theorem~\ref{Hom}(1) $K$ should be an L-space knot, 
and it has no essential $1$--string tangle decomposition \cite{Kr}. 
Hence $n' \ge 2$. 
This means that $S$ intersects $P(K)$ at least $2 \omega n' \ge 8$ times, 
a contradiction. 
\end{proof}

Since a torus knot has no essential $n$--string tangle decomposition for all $n$ \cite{GR, Tsau} we have: 

\begin{corollary}
\label{L-space knots_n-string prime}
Suppose that Conjecture~\ref{L-slopes} is true. 
Then Conjecture~\ref{n-string prime} is true whenever hyperbolic L-space knots have no essential tangle decomposition. 
\end{corollary}

\end{document}